\documentclass[11 point, a4paper, reqno]{amsart}
\usepackage[english]{babel}
\usepackage{amsthm}
\usepackage{graphicx}
\usepackage{amsfonts}
\usepackage{eufrak}
\usepackage{amsopn}
\usepackage{amsmath, amsfonts, amssymb, amscd, amsthm}
\usepackage{yhmath}
\usepackage{setspace}
\usepackage{mathtools}
\usepackage{multirow}
\usepackage{caption}
\usepackage{graphicx}
\usepackage{lscape}
\usepackage[all]{xy}

\newtheorem{theorem}{Theorem}[section]
\newtheorem{lemma}[theorem]{Lemma}
\newtheorem{proposition}[theorem]{Proposition}
\newtheorem{corollary}[theorem]{Corollary}
\newtheorem{conjecture}[theorem]{Conjecture}
\theoremstyle{definition}

\theoremstyle{remark}
\newtheorem{remark}[theorem]{Remark}
\theoremstyle{example}

\theoremstyle{note}

\numberwithin{equation}{section}

\DeclareMathOperator{\GL}{GL}

\DeclareMathOperator{\Ker}{Ker}

\DeclareMathOperator{\Image}{Im}
\DeclareMathOperator{\Span}{Span}
\DeclareMathOperator{\ind}{ind}
\DeclareMathOperator{\M}{M}

\DeclareMathOperator{\tr}{Tr}

\DeclareMathOperator{\Rank}{Rank}
\DeclareMathOperator{\Type}{Type}

\begin{document}
\title{On a twisted Jacquet module of $\GL(6)$ over a finite field}
\author{Kumar Balasubramanian} 
\thanks{Research of Kumar Balasubramanian is supported by the SERB grant: MTR/2019/000358.}
\address{Kumar Balasubramanian\\
Department of Mathematics\\
IISER Bhopal\\
Bhopal, Madhya Pradesh 462066, India}
\email{bkumar@iiserb.ac.in}

\author{Himanshi Khurana*} \thanks{* indicates corresponding author}
\address{Himanshi Khurana\\
Department of Mathematics\\
IISER Bhopal\\
Bhopal, Madhya Pradesh 462066, India}
\email{himanshi18@iiserb.ac.in}


\maketitle

\begin{abstract} Let $F$ be a finite field and $G=\GL(6,F)$. In this paper, we explicitly describe a certain twisted Jacquet module of an irreducible cuspidal representation of $G$.
\end{abstract}

\section{Introduction}

Let $F$ be a finite field and $G=\GL(n,F)$. Let $P$ be a parabolic subgroup of $G$ with Levi decomposition $P=MN$. Let $\pi$ be any irreducible finite dimensional complex representation of $G$ and $\psi$ be an irreducible representation of $N$. Let $\pi_{N, \psi}$ be the sum of all irreducible representations of $N$ inside $\pi$, on which $\pi$ acts via the character $\psi$. It is easy to see that $\pi_{N, \psi}$ is a representation of the subgroup $M_{\psi}$ of $M$, consisting of those elements in $M$ which leave the isomorphism class of $\psi$ invariant under the inner conjugation action of $M$ on $N$. The space $\pi_{N, \psi}$ is called the \textit{twisted Jacquet module} of the representation $\pi$. It is an interesting question to understand for which irreducible representations $\pi$, the twisted Jacquet module $\pi_{N, \psi}$ is non-zero and to understand its structure as a module for $M_{\psi}$. \\ 

In an earlier work of ours \cite{KumHim}, inspired by the work of Prasad in \cite{Dip[1]}, we studied the structure of a certain twisted Jacquet module of a cuspidal representation of $\GL(4,F)$. In this paper, we continue our study of the twisted Jacquet module for a cuspidal representation of $\GL(6,F)$. We refer the reader to Section 1 in \cite{KumHim} for a more elaborate introduction and the motivation to study the problem. \\

Before we state our result, we set up some notation. Let $G=\GL(6,F)$ and $P$ be the maximal parabolic subgroup of $G$ with Levi decomposition $P=MN$, where $M\simeq \GL(3,F) \times \GL(3,F)$ and $N\simeq \M(3,F)$.  We write $F_{6}$ for the unique field extension of $F$ of degree $6$. Let $\psi_{0}$ be a fixed non-trivial additive character of $F$.
Let \[A=\begin{bmatrix} 0 & 0 & 1\\ 0 & 0 & 0\\
0 & 0 & 0\end{bmatrix}\] and $\psi_{A}: N\rightarrow \mathbb{C}^{\times}$ be the character of $N$ given by
\begin{equation}\label{definition of psi_A}\psi_{A}\left ( \begin{bmatrix} 1 & X \\ 0 & 1\end{bmatrix}\right ) = \psi_{0}(\tr(AX)). \end{equation}
Let $H_{A}=M_{1} \times M_{2}$ where $M_{1}$ is the Mirabolic subgroup of $\GL(3,F)$ and $M_{2}=w_{0}M_{1}^{\top}w_{0}^{-1}$ where $w_{0}= \begin{bmatrix} 0 & 0 & 1 \\
0 & 1 & 0\\ 1 & 0 & 0\end{bmatrix}$. Let $U$ be the subgroup of unipotent matrices in $\GL(6,F)$ and $U_{A}= U\cap H_{A}$. Clearly, we have $U_{A}\simeq U_{1} \times U_{2}$ where $U_{1}$ and $U_{2}$ are the upper triangular unipotent subgroups of $\GL(3,F)$. For $k=1,2$, let $\mu_{k}: U_{k}\rightarrow \mathbb{C}^{\times}$ be the non-degenerate character of $U_{k}$ given by 
\[\mu_{k} \left ( \begin{bmatrix}
 1 & x_{12} & x_{13}\\
 0 &   1   & x_{23} \\
 0 &   0   &     1
 \end{bmatrix} \right )= \psi_0(x_{12}+ x_{23}).\]
Let $\mu: U_{A}\rightarrow \mathbb{C}^{\times}$ be the character of $U_{A}$ given by \[\mu(u)=\mu_{1}(u_{1})\mu_{2}(u_{2})\] where $u= \begin{bmatrix} u_{1} & 0 \\ 0 & u_{2}\end{bmatrix}$. \\

\begin{theorem}
Let $\theta$ be a regular character of $F_6^{\times}$ and $\pi=\pi_{\theta}$ be an irreducible cuspidal representation of $G$. Then 
\[\pi_{N,\psi_A} \simeq \theta|_{F^{\times}} \otimes \ind_{U_{A}}^{H_{A}}\mu\]
as $M_{\psi_A}$ modules.
\end{theorem}

We establish the above isomorphism by explicitly calculating the characters of $\pi_{N,\psi_{A}}$ and $\theta|_{F^{\times}} \otimes \ind_{U_A}^{H_A}(\mu)$, and showing that they are equal at any arbitrary element of $M_{\psi_A}$. \\

The calculation of the twisted Jacquet module for $\GL(4,F)$, did not provide us with much insight to predict the structure of the twisted Jacquet module for $\GL(2n,F)$. This motivated us to study the problem for $\GL(6,F)$ to get a better understanding of its structure. Based on our computations in these particular cases, we formulate the following conjecture for $\GL(2n,F)$. \\

Let $F_{n}$ be the unique field extension of $F$ of degree $n$ and let $A=\begin{bmatrix} 0 & \cdots & 1\\ \vdots & \adots & \vdots \\ 0 & \cdots & 0\end{bmatrix}\in \M(n,F)$. We let $H_{A}=M_{1}\times M_{2}$, $U_{A}=U_{1}\times U_{2}$ where  $M_{1}, M_{2}, U_{1}, U_{2}$ are appropriate subgroups of $\GL(n,F)$ as defined earlier. 

\begin{conjecture} Let $\theta$ be a regular character of $F_n^{\times}$ and $\pi=\pi_{\theta}$ be an irreducible cuspidal representation of $G$. Then 
\[\pi_{N,\psi_A} \simeq \theta|_{F^{\times}} \otimes \ind_{U_{A}}^{H_{A}}\mu\]
as $M_{\psi_A}$ modules.
\end{conjecture}

The calculations involved in the $\GL(6,F)$ case are much more involved than in the case of $\GL(4,F)$ and we hope that some of these calculations may be useful in the general case. It is also an interesting problem to study the structure of $\pi_{N,\psi_{A}}$ in the case when the finite field is replaced with a p-adic field. Our hope is that understanding the problem completely for the finite group case might help in understanding the problem in the p-adic case. We hope to study these problems in future.

\section{Preliminaries}
In this section, we mention some preliminary results that we need in our paper.

\subsection{Character of a Cuspidal Representation} Let $F$ be the finite field of order $q$ and $G=\GL(m,F)$. Let $F_m$ be the unique field extension of $F$ of degree $m$. A character $\theta$ of $F^{\times}_{m}$ is called a ``regular'' character, if under the action of the Galois group of $F_{m}$ over $F$, $\theta$ gives rise to $m$ distinct characters of $F^{\times}_{m}$. It is a well known fact that the cuspidal representations of $\GL(m,F)$ are parametrized by the regular characters of $F_{m}^{\times}$. To avoid introducing more notation, we mention below only the relevant statements on computing the character values that we have used. We refer the reader to Section 6 in \cite{Gel[1]} for more precise statements on computing character values.

\begin{theorem}\label{character value of cuspidal representation (Gelfand)}
Let $\theta$ be a regular character of $F^{\times}_{m}$. Let $\pi=\pi_{\theta}$ be an irreducible cuspidal representation of $\GL(m,F)$ associated to $\theta$. Let $\Theta_{\theta}$ be its character. If $g\in \GL(m,F)$ is such that the characteristic polynomial of $g$ is not a power of a polynomial irreducible over $F$. Then, we have  \[\Theta_\theta(g)=0. \]
\end{theorem}

\begin{theorem}\label{character value of cuspidal representation (Dipendra)} Let $\theta$ be a regular character of $F^{\times}_{m}$. Let $\pi=\pi_{\theta}$ be an irreducible cuspidal representation of $\GL(m,F)$ associated to $\theta$. Let $\Theta_{\theta}$ be its character. Suppose that $g=s.u$ is the Jordan decomposition of an element $g$ in $\GL(m,F)$. If $\Theta_{\theta}(g)\neq 0$, then the semisimple element $s$ must come from $F_{m}^{\times}$. Suppose that $s$ comes from $F_{m}^{\times}$. Let $z$ be an eigenvalue of $s$ in $F_{m}$ and let $t$ be the dimension of the kernel of $g-z$ over $F_{m}$. Then
\[\Theta_{\theta}(g)=(-1)^{m-1}\bigg[\sum_{\alpha=0}^{d-1}\theta(z^{q^{\alpha}})\bigg ](1-q^{d})(1-(q^{d})^{2})\cdots (1-(q^{d})^{t-1}). \]
where $q^{d}$ is the cardinality of the field generated by $z$ over $F$, and the summation is over the distinct Galois conjugates of $z$.
\end{theorem}

See Theorem 2 in \cite{Dip[1]} for this version.

\subsection{Twisted Jacquet Module}

In this section, we recall the character and the dimension formula of the twisted Jacquet module of a representation $\pi$. \\

Let $G=\GL(k,F)$ and $P=MN$ be a parabolic subgroup of $G$. Let $\psi$ be a character of $N$. For $m\in M$, let $\psi^{m}$ be the character of $N$ defined by $\psi^{m}(n)=\psi(mnm^{-1})$. Let \[V(N,\psi)= \Span_{\mathbb{C}} \{\pi(n)v-\psi(n)v \mid n \in N, v \in V \} \]
and \[M_{\psi} = \{m \in M \mid {\psi}^{m}(n)=\psi(n) , \forall n \in N \} . \]
Clearly, $M_{\psi}$ is a subgroup of $M$ and it is easy to see that $V(N,\psi)$ is an $M_{\psi}$-invariant subspace of $V$. Hence, we get a representation  $(\pi_{N,\psi},V/V(N,\psi))$ of $M_{\psi}$. We call $(\pi_{N,\psi}, V/V(N,\psi))$ the twisted Jacquet module of $\pi$ with respect to $\psi$. We write $\Theta_{N, \psi}$ for the character of $\pi_{N, \psi}$.

\begin{proposition}
Let $(\pi,V)$ be a representation of $\GL(k,F)$  and $\Theta_\pi$ be the character of $\pi$. We have
\[\Theta_{N,\psi}(m) = \frac{1}{|N|}\sum_{n \in N}\Theta_{\pi}(mn)\overline{\psi(n)}.\]
\end{proposition}
We refer the reader to Proposition 2.3 in \cite{KumHim} for a proof. 

\begin{remark}
Taking $m=1$, we get the dimension of $\pi_{N, \psi}$. To be precise, we have
\[\dim_{\mathbb{C}}(\pi_{N,\psi}) =  \frac{1}{|N|}\sum_{n \in N}\Theta_{\pi}(n)\overline{\psi(n)}.\]
\end{remark}

\subsection{Character of the induced representation} In this section, we recall the character formula for the induced representation of a group $G$. For a proof, we refer the reader to Chapter 3, Theorem 12 in \cite{Ser}. \\

\begin{proposition}
Let $G$ be a finite group and $H$ be a subgroup of $G$. Let $(\pi,V)$ be a representation of $H$ and $\chi_\pi$ be the character of $\pi$. Then for each $s \in G$, the character of $\ind_{H}^{G}(\pi)$ is given by
\[\chi_{\ind_H^{G}(\pi)}(s) = \frac{1}{|H|} \sum_{\substack{t \in G \\ t^{-1}st \in H }}\chi_\pi(t^{-1}st). \]
\end{proposition}

\section{Dimension of the Twisted Jacquet Module}
 
Let $\pi=\pi_{\theta}$ be an irreducible cuspidal representation of $G$ corresponding to the regular character $\theta$ of $F_{6}^{\times}$ and $\Theta_{\theta}$ be its character. Throughout, we write $\M(n,m,r,q)$ for the set of $n\times m$ matrices of rank $r$ over the finite field $F=F_{q}$. In this section, we calculate the dimension of $\pi_{N,\psi_{A}}$. Before we continue, we record some preliminary lemmas that we need. \\

\begin{lemma} Let $r\in \{0, 1, 2 ,3\}$ and $X\in \M(3,3,r,q)$. We have
\[\Theta_{\theta}\left (\begin{bmatrix} 1 & X \\ 0 & 1 \end{bmatrix}\right )=   \begin{cases}
\hspace{0.3 cm} (q-1)(q^{2}-1)(q^{3}-1)(q^4-1)(q^5-1), \hspace{0.3 cm} & \text{if} \hspace{0.2 cm}  $r=0 $\\
-(q-1)(q^{2}-1)(q^3-1)(q^4-1), \hspace{0.3 cm} & \text{if} \hspace{0.2 cm} $r=1$ \\
\hspace{0.3 cm} (q-1)(q^2-1)(q^3-1), & \text{if} \hspace{0.2 cm} $r=2$  \\
-(q-1)(q^2-1), & \text{if} \hspace{0.2 cm} $r=3$  \\
\end{cases}\]
\end{lemma}

\begin{proof} The result follows from Theorem~\ref{character value of cuspidal representation (Dipendra)} above. \\
\end{proof}

Let \[X= \begin{bmatrix}
a & d & g\\
b & e & h\\
c & f & k \end{bmatrix}, A= \begin{bmatrix}
1 & 0 & 0\\
0 & 0 & 0\\
0 & 0 & 0 \end{bmatrix}, AX= \begin{bmatrix} a & d & g\\
0 & 0 & 0\\
0 & 0 & 0 \end{bmatrix}.\]

For $\alpha \in F$ and $r \in \{0,1,2,3\}$, consider the subset ${Y^{\alpha}_{3,r}}$ of $\M(3,F)$ given by 
\[ {Y^{\alpha}_{3,r}}=\{ X \in \M(3,F) \mid \Rank(X)=r, \tr(AX)= \alpha \}.\]

\begin{lemma}\label{equal cardinality} Let $r\in \{1,2,3\}$ and $\alpha, \beta \in F^{\times}$. Then we have
\[\# Y_{3,r}^{\alpha} = \# Y_{3,r}^{\beta}. \]
\end{lemma}
\begin{proof} Consider the map $\phi: Y_{3,r}^{\alpha} \to Y_{3,r}^{\beta}$ given by \[\phi(X)=\alpha^{-1}\beta X.\] 
Suppose that $\phi(X)=\phi(Y)$. Since $\alpha^{-1}\beta \neq 0$, it follows that $\phi$ is injective. For $Y \in Y_{3,r}^{\beta}$, let $X=\alpha \beta^{-1} Y$. Clearly, we have $\tr(AX)=\alpha$ and $\Rank(X)=\Rank(Y)=r$. Thus $\phi$ is surjective and hence the result.
\end{proof}

\begin{theorem}\label{dimension calculation} Let $\theta$ be a regular character of $F_{6}^{\times}$ and $\pi=\pi_{\theta}$ be an irreducible cuspidal representation of $\GL(6,F)$. We have \[\dim_{\mathbb{C}}(\pi_{N, \psi_{A}})= (q-1)^{2}(q^2-1)^{2}.\]
\end{theorem}
\begin{proof}
It is easy to see that the dimension of $\pi_{N,\psi_{A}}$ is given by
\begin{equation}\label{dimension formula} \dim_{\mathbb{C}}(\pi_{N,\psi_{A}})= \frac{1}{q^{9}} \sum_{X\in \M(3,F)}\Theta_{\theta}\left(\begin{bmatrix} 1 & X \\ 0 & 1\end{bmatrix}\right)\overline{\psi_{0}(\tr(AX))}. \end{equation}

We calculate the following sums

\begin{enumerate}
\item[a)] $S_{1} = \displaystyle \sum_{X\in \M(3,3,0,q)} \Theta_{\theta}\left (\begin{bmatrix} 1 & X \\ 0 & 1 \end{bmatrix}\right )\overline{\psi_{0}(\tr(AX))}$ \\
\item[b)] $S_{2} = \displaystyle \sum_{\substack{X\in \M(3,3,1,q) \\ \tr(AX)=0 }}\Theta_{\theta}\left (\begin{bmatrix} 1 & X \\ 0 & 1 \end{bmatrix}\right )\overline{\psi_{0}(\tr(AX))}\quad   +  \displaystyle \sum_{\substack{X\in \M(3,3,1,q) \\ \tr(AX)=\alpha\neq 0 }}\Theta_{\theta}\left (\begin{bmatrix} 1 & X \\ 0 & 1 \end{bmatrix}\right )\overline{\psi_{0}(\tr(AX))}$ \\
\item[c)] $S_{3} = \displaystyle \sum_{\substack{X\in \M(3,3,2,q) \\ \tr(AX)=0 }}\Theta_{\theta}\left (\begin{bmatrix} 1 & X \\ 0 & 1 \end{bmatrix}\right )\overline{\psi_{0}(\tr(AX))}\quad  +  \displaystyle \sum_{\substack{X\in \M(3,3,2,q) \\ \tr(AX)=\alpha\neq 0 }}\Theta_{\theta}\left (\begin{bmatrix} 1 & X \\ 0 & 1 \end{bmatrix}\right )\overline{\psi_{0}(\tr(AX))}$ \\
\item[d)] $S_{4} = \displaystyle \sum_{\substack{X\in \M(3,3,3,q) \\ \tr(AX)=0 }}\Theta_{\theta}\left (\begin{bmatrix} 1 & X \\ 0 & 1 \end{bmatrix}\right )\overline{\psi_{0}(\tr(AX))}\quad  +  \displaystyle \sum_{\substack{X\in \M(3,3,3,q) \\ \tr(AX)=\alpha\neq 0 }}\Theta_{\theta}\left (\begin{bmatrix} 1 & X \\ 0 & 1 \end{bmatrix}\right )\overline{\psi_{0}(\tr(AX))}$ 
\end{enumerate}
separately to compute the dimension of $\pi_{N, \psi_{A}}$. \\

For a fixed $r \in \{0,1,2,3\}$ and $\alpha \in \{0,1\}$, we find a partition of $Y_{3,r}^{\alpha}$ into certain subsets, and compute the cardinality of each of these subsets to find the cardinality of $Y_{3,r}^{\alpha}$. We record the necessary information in the tables below.\\

For $a)$, we clearly have
\begin{align*}
S_{1} &= \Theta_{\theta} \left (\begin{bmatrix} 1 & X \\ 0 & 1 \end{bmatrix}\right ) \overline{\psi_{0}(0)} \\
&= (q-1)(q^{2}-1)(q^{3}-1)(q^4-1)(q^5-1). 
\end{align*}
\vspace{-0.5 cm}

For $b)$, we have
\begin{center}
\captionof{table}{$\Rank(X)=1$}
\label{rk(X)=1 type}
\scriptsize
 \begin{tabular}{| c | c | c | c |}
 \hline 
 & & & \\
 Partition of $Y_{3,1}^{0}$& Cardinality & Partition of $Y_{3,1}^{1}$ & Cardinality\\
 & & &\\
  
    \hline
    & &    & \\
    $\left \{ \begin{bmatrix} 0 & 0 &0  \\ b & \lambda b & \beta b\\
c & \lambda c & \beta c \end{bmatrix} \right \}$    & $(q^2-1)q^2$  &$  \left \{ \begin{bmatrix} 1 & \lambda & \beta\\
         b & \lambda b & \beta b \\
         c & \lambda c & \beta c
         \end{bmatrix}  \right \} $ & $q^4$\\
 & & &\\
    \hline

         &     &      &            \\
          $ \left \{ \begin{bmatrix} 0 & d & \lambda d\\
         0 & e & \lambda e\\
         0 & f & \lambda f
         \end{bmatrix} \right \}$ & $(q^3-1)q$ & - & - \\
         & & & \\
    \hline

         &     &         &            \\
         $\left \{ \begin{bmatrix} 0 & 0 & g\\
         0 & 0 & h\\
         0 & 0 & k
         \end{bmatrix} \right \} $ & $q^3-1$ & - &  -\\  
         & &  &\\
    \hline
    
    \end{tabular} 
\end{center}

\vspace{-0.2 cm}
Thus, 
\begin{align*}
S_{2} &= \Theta_{\theta} \left (\begin{bmatrix} 1 & X \\ 0 & 1 \end{bmatrix}\right ) \left(\displaystyle \sum_{\substack{X\in \M(3,3,1,q) \\ \tr(AX)=0 }}\overline{\psi_{0}(0)}   \quad +  \displaystyle \sum_{\substack{X\in \M(3,3,1,q) \\ \tr(AX)=\alpha\neq 0 }}\overline{\psi_{0}(\alpha)} \right ) \\\\
&= \Theta_{\theta} \left (\begin{bmatrix} 1 & X \\ 0 & 1 \end{bmatrix}\right ) \left ( \# Y_{3,1}^{0} - \# Y_{3,1}^{1}\right )\\ \\
&= -(q-1)(q^{2}-1)(q^{3}-1)(q^{4}-1)((q^{2}-1)q^{2}+ (q^{3}-1)q + (q^{3}-1)- q^{4})\\\\
&= -(q-1)^2(q^{2}-1)^2(q^8+2q^7+2q^6+q^5-2q^4-3q^3-4q^2-2q-1). \\
\end{align*}
For $d)$, we have
\begin{center}
\captionof{table}{$\Rank(X)=3$}
\label{rk(X)=3 type}
\scriptsize
 \begin{tabular}{| c | c | c | c |}
 \hline 
 & & & \\
 Partition of $Y_{3,3}^{0}$& Cardinality & Partition of $Y_{3,3}^{1}$ & Cardinality\\
 & & &\\
  
    \hline
    & &    & \\
   $ \left \{ \begin{bmatrix} 0 & d & g  \\ b & e & h\\
c & f & k \end{bmatrix}  \right \}$    & $(q^2-1)(q^3-q)(q^3-q^2)$  & $  \left \{ \begin{bmatrix} 1 & d & g\\
         b & e  & h \\
         c & f & k
         \end{bmatrix}  \right \} $ & $q^2(q^3-q)(q^3-q^2)$\\
 & & &\\
    \hline

    \end{tabular} 
\end{center}
Thus,
\begin{align*}
S_{4} &= \Theta_{\theta} \left (\begin{bmatrix} 1 & X \\ 0 & 1 \end{bmatrix}\right ) \left(\displaystyle \sum_{\substack{X\in \M(3,3,3,q) \\ \tr(AX)=0 }}\overline{\psi_{0}(0)}   \quad +  \displaystyle \sum_{\substack{X\in \M(3,3,3,q) \\ \tr(AX)=\alpha\neq 0 }}\overline{\psi_{0}(\alpha)} \right ) \\\\
&= \Theta_{\theta} \left (\begin{bmatrix} 1 & X \\ 0 & 1 \end{bmatrix}\right ) \left ( \# Y_{3,3}^{0} - \# Y_{3,3}^{1}\right )\\ \\
&= -(q-1)(q^{2}-1)((q^{2}-1)(q^{3}-q)(q^{3}-q^{2})- q^{2}(q^{3}-q)(q^{3}-q^{2}))\\\\
&= (q-1)^2(q^2-1)^2q^3.
\end{align*}

For $c)$, we let $X'=\begin{bmatrix} e & h \\f & k \end{bmatrix}$. For $\alpha\in \{0,1\}$, we partition the set $Y_{3,2}^{\alpha}$ according to the rank of $X'$ and count the cardinalities of each of these subsets. For $\Rank(X')\in \{0,1,2\}$ and $\alpha\in \{0,1\}$ and we record the cardinality of such subsets of $Y_{3,2}^{\alpha}$ in the following tables. \\

\begin{center}
\captionof{table}{$\Rank(X)=2, \Rank(X')=0$}
\label{rk(X)=2,rk(X')=0 type}
\scriptsize
 \begin{tabular}{| c | c | c | c |}
 \hline 
 & & & \\
 Partition of $Y_{3,2}^{0}$ & Cardinality & Partition of $Y_{3,2}^{1}$ & Cardinality\\
 & & &\\
  
    \hline
    & &    & \\
   $B_1=  \left \{ \begin{bmatrix} 0 & d & g  \\ b & 0 & 0\\
c & 0 & 0 \end{bmatrix} \,\middle\vert\, (d,g) \neq (0,0)  \right \}$    & $ (q^2-1)^2$ & $B_2 =  \left \{ \begin{bmatrix} 1 & d & g\\
         b & 0 &  0 \\
         c & 0 & 0
                  \end{bmatrix}  \,\middle\vert\, (d,g) \neq (0,0)  \right \} $ & $(q^2-1)^2$\\
 & & &\\
    \hline

    \end{tabular} 
\end{center}

\begin{center}
\captionof{table}{$\Rank(X)=2, \Rank(X')=2$}
\label{rk(X)=2,rk(X')=2 type}
\scriptsize
\resizebox{1.25\textwidth}{!}{
 \begin{tabular}{| c | c | c | c |}
 \hline 
 & & &  \\
 Partition of $Y_{3,2}^0$ & Cardinality & Partition of $Y_{3,2}^1$ & Cardinality \\
 & & &\\
  
    \hline
    & & &\\
   $C_1=  \left \{ \begin{bmatrix} 0 & 0 & 0\\
   \lambda e+\beta h  & e & h  \\ \lambda f+\beta k & f & k\\
 \end{bmatrix}  \right \}$    & $ (q^2-1)(q^2-q)(q^2)$ & $C_4 =  \left \{ \begin{bmatrix} 1  & d & 0\\
         d^{-1}e+\beta h & e &  h \\
         d^{-1}f+\beta k & f & k
                  \end{bmatrix}, \begin{bmatrix} 1 & 0 & g\\
                 
         \lambda e+ g^{-1} h & e &  h \\
         \lambda f+g^{-1}k & f & k
                  \end{bmatrix} \right \} $ & $2(q^2-1)(q^2-q)^2$\\
& & & \\
\hline
 & & & \\
   $C_2=  \left \{ \begin{bmatrix}0  & d & 0  \\ \beta h & e & h\\
\beta k & f & k \end{bmatrix}, \begin{bmatrix}0  & 0 & g  \\ \beta e & e & h\\
\beta f & f & k \end{bmatrix} \right \}$    & $ 2(q^2-1)(q^2-q)^2$ & $C_5 =  \left \{ \begin{bmatrix} 1  & d & g\\
         d^{-1}(1-\beta g)e+\beta h & e &  h \\
         d^{-1}(1-\beta g)f+\beta k & f & k
                  \end{bmatrix} \right \} $ & $(q^2-1)(q^3-q^2)(q-1)^2$\\
& & &\\

\hline
 & & & \\
   $C_3=  \left \{ \begin{bmatrix}0  & d & g  \\ \beta(-gd^{-1}e+h) & e & h\\
\beta (-gd^{-1}f+ k) & f & k \end{bmatrix} \right \}$    & $ (q^2-1)(q^2-q)^2(q-1)$ &- & -\\
 & & &\\
 \hline 
    \end{tabular} }
\end{center}

\begin{center}
\captionof{table}{$\Rank(X)=2,\Rank(X')=1$}
\label{rk(X)=2,rk(X')=1 type}
\scriptsize
\resizebox{1.25\textwidth}{!}{
 \begin{tabular}{| c | c | c | c |}
 \hline 
 & & & \\
 Partition of $Y_{3,2}^{0}$ & Cardinality & Partition of $Y_{3,2}^1$ & Cardinality\\
 & & &\\
  
    \hline
    & &    & \\
   $E_1=  \left \{ \begin{bmatrix} 0 & 0 & 0 \\ b & e & 0\\
c & f & 0 \end{bmatrix},  \begin{bmatrix} 0 & 0 & 0 \\ b & 0 & h\\
c & 0 & k \end{bmatrix} \right \}$    & $2(q^2-1)(q^2-q)$  & $F_1=  \left \{ \begin{bmatrix} 1 & 0 & 0 \\ b & 0 & h\\
c & 0 & k \end{bmatrix}, \begin{bmatrix} 1 & 0 & 0 \\ b & e & 0\\
c & f & 0 \end{bmatrix}  \right \} $ & $2(q^2-1)q^2$\\
 & & &\\
    \hline

         &     &      &            \\
          $E_2= \left \{ \begin{bmatrix} 0 & d & 0\\
          b & e & 0\\
         c & f & 0
         \end{bmatrix}, \begin{bmatrix} 0 & 0 & g\\
         b  & 0 & h\\
         c & 0 & k
         \end{bmatrix}\right \}$ & $2(q^2-1)^2(q-1)$ & $F_2=  \left \{ \begin{bmatrix} 1 & d & 0 \\ \beta h & 0 & h\\
\beta k & 0 & k \end{bmatrix}, \begin{bmatrix} 1 & 0 & g \\ \beta e & e & 0\\
\beta f & f & 0 \end{bmatrix}\right \} $ & $2(q^2-1)(q^2-q)$\\
         & & & \\
    \hline

         &     &         &            \\
         $E_3*= \left \{ \begin{bmatrix} 0 & 0 & g\\
         \lambda e  & e & 0\\
         \lambda f & f & 0
         \end{bmatrix},  \begin{bmatrix} 0 & d & 0\\
          \lambda h & 0 & h\\
         \lambda k & 0 & k
         \end{bmatrix}\right \} $ & $2(q^2-1)(q^2-q)$ & $F_3=  \left \{ \begin{bmatrix} 1 & 0 & g \\ b & 0 & h\\
c & 0 & k \end{bmatrix}  \,\middle\vert\, (b,c) \neq (g^{-1}h,g^{-1}k)  \right \} \cup \left \{ \begin{bmatrix} 1 & d & 0 \\ b & e & 0\\
c & f & 0 \end{bmatrix}  \,\middle\vert\, (b,c) \neq (d^{-1}e,d^{-1}f)  \right \} $ & $2(q^2-1)^2(q-1)$\\  
         & &  &\\
    \hline
     &     &         &            \\
         $E_4 = \left \{ \begin{bmatrix} 0 & d & g\\
         -d^{-1}\beta ge & e & 0\\
         -d^{-1}\beta g f & f & 0
         \end{bmatrix}, \begin{bmatrix} 0 & d & g\\
         -g^{-1}\beta dh & 0 & h\\
         -g^{-1}\beta dk & 0 & k
         \end{bmatrix} \right \} $ & $2(q^3-q)(q-1)^2$ & $F_4=  \left \{ \begin{bmatrix} 1 & d & g \\ \beta h & 0 & h\\
\beta k & 0 & k \end{bmatrix}, \begin{bmatrix} 1 & d & g \\ \beta e & e & 0\\
\beta f & f & 0 \end{bmatrix} \right \} $ & $2q(q^2-1)(q-1)^2$\\
         & &  &\\
    \hline
     &     &         &            \\
         $E_5 = \left \{ \begin{bmatrix} 0 & 0 & 0\\
         b & \lambda h & h\\
         c & \lambda k & k
         \end{bmatrix}  \,\middle\vert\, \lambda \neq 0 \right \} $ & $(q^2-1)(q-1)(q^2-q)$ & $F_5=  \left \{ \begin{bmatrix} 1 & 0 & 0 \\ b & e & \lambda e\\
c & f & \lambda f \end{bmatrix}  \,\middle\vert\, \lambda \neq 0 \right \} $ & $(q^2-1)q^2(q-1)$\\ 
         & &  &\\
    \hline
      &     &         &            \\
         $E_6 = \left \{ \begin{bmatrix} 0 & d & 0\\
         \beta h & \lambda h & h\\
         \beta k & \lambda k & k
         \end{bmatrix},\begin{bmatrix} 0 & 0 & g\\
         \beta h & \lambda h & h\\
         \beta k & \lambda k & k
         \end{bmatrix}  \,\middle\vert\, \lambda \neq 0 \right \} $ & $2q(q^2-1)(q-1)^2$ & $F_6=  \left \{ \begin{bmatrix} 1 & d & 0 \\ d^{-1}e+\beta \lambda e & e & \lambda e\\
d^{-1}f+\beta\lambda f & f & \lambda f \end{bmatrix}, \begin{bmatrix} 1 & 0 & g \\ \delta e+ g^{-1} \lambda e & e & \lambda e\\
\delta f+g^{-1}\lambda f & f & \lambda f \end{bmatrix}  \,\middle\vert\, \lambda \neq 0 \right \} $ & $2q(q^2-1)(q-1)^2$\\
         & &  &\\
    \hline
     &     &         &            \\
         $E_7 = \left \{ \begin{bmatrix} 0 & d & g\\
         -d^{-1}\beta g \lambda h+\beta h & \lambda h & h\\
         -d^{-1}\beta g \lambda k+\beta k & \lambda k & k
         \end{bmatrix} \mid d \neq \lambda g, \lambda \neq 0 \right \} $ & $(q^2-1)(q-1)^2(q^2-2q)$ & $F_7=  \left \{ \begin{bmatrix} 1 & d & g \\d^{-1}(1-\beta g)e+\beta \lambda e & e & \lambda e\\
d^{-1}(1-\beta g)f+\beta \lambda f & f & \lambda f \end{bmatrix}  \,\middle\vert\, g \neq \lambda d, \lambda \neq 0 \right \} $ & $(q^2-1)(q-1)^2(q^2-2q)$\\ 
         & &  &\\
    \hline
     &     &         &            \\
         $E_8 = \left \{ \begin{bmatrix} 0 & \lambda g & g\\
         b & \lambda h & h\\
         c & \lambda k & k
         \end{bmatrix} \lambda \neq 0 \right \} $ & $(q^2-1)^2(q-1)^2$ & $F_8=  \left \{ \begin{bmatrix} 1 & d & \lambda d \\ b & e & \lambda e\\
c &  f & \lambda f  \end{bmatrix}  \,\middle\vert\, \lambda \neq 0  \right \} $ & $(q^2-1)^2(q-1)^2$\\ 
         & &  &\\
    \hline
    \end{tabular} }
\end{center}
We have, 
\[Y_{3,2}^0= B_1 \bigsqcup_{i=1}^3 C_i \bigsqcup_{j=1}^8 E_j \]
and
\[Y_{3,2}^1= B_2 \bigsqcup_{i=4}^5 C_i \bigsqcup_{j=1}^8 F_j.\]
Thus,
\begin{align*}
S_{3} &= \Theta_{\theta} \left (\begin{bmatrix} 1 & X \\ 0 & 1 \end{bmatrix}\right ) \left(\displaystyle \sum_{\substack{X\in \M(3,3,2,q) \\ \tr(AX)=0 }}\overline{\psi_{0}(0)}   \quad +  \displaystyle \sum_{\substack{X\in \M(3,3,2,q) \\ \tr(AX)=\alpha\neq 0 }}\overline{\psi_{0}(\alpha)} \right ) \\\\
&= \Theta_{\theta} \left (\begin{bmatrix} 1 & X \\ 0 & 1 \end{bmatrix}\right ) \left ( \# Y_{3,2}^{0} - \# Y_{3,2}^{1}\right )\\ \\
&= (q-1)(q^2-1)(q^3-1)(q^6-q^5-2q^4+q^2+q)\\\\
&= (q-1)^2(q^2-1)^2(q^6-q^4-3q^3-2q^2-q).
\end{align*}
From ~\eqref{dimension formula}, it follows that
\begin{align*}
\dim_{\mathbb{C}}(\pi_{N,\psi_{A}}) & = \frac{1}{q^{9}}\{S_{1} + S_{2} + S_{3}+S_{4]}\} \\
&= \frac{1}{q^{9}}(q-1)^2(q^2-1)^2q^9\\
&= (q-1)^{2}(q^2-1)^2.
\end{align*}
\end{proof}

\begin{remark} Suppose that $B=Aw_{0}$. It is easy to see that $\Theta_{N,\psi_{A}}\left(\begin{bmatrix} m_{1} & 0 \\ 0 & m_{2}\end{bmatrix}\right)= \Theta_{N,\psi_{B}}\left( \begin{bmatrix} w_{0}m_{1}w_{0} & 0 \\ 0 & m_{2}\end{bmatrix} \right)$. Thus we have that $\dim(\pi_{N,\psi_{A}})=\dim(\pi_{N, \psi_{B}})$. 
\end{remark}

\section{Main Theorem}
In this section, we prove the main result of this paper. Hereafter, we take $$A=\begin{bmatrix} 0 & 0 & 1 \\ 0 & 0 & 0 \\ 0 & 0 & 0\end{bmatrix}.$$For the sake of completeness, we recall the statement below. 
\begin{theorem}
Let $\theta$ be a regular character of $F_6^{\times}$ and $\pi=\pi_{\theta}$ be an irreducible cuspidal representation of $G$. Then 
\[\pi_{N,\psi_A} \simeq \theta|_{F^{\times}} \otimes \ind_{U_{A}}^{H_{A}}\mu\]
as $M_{\psi_A}$ modules.
\end{theorem}

The key idea of the proof is to compute the characters of the representations $\rho=\theta|_{F^{\times}} \otimes \ind_{U_{A}}^{H_{A}}\mu$ and $\pi_{N,\psi_{A}}$ and show that they are equal at any arbitrary element in $M_{\psi_{A}}$. Before we continue, we set up some notation and record a few lemmas that we need.

\begin{lemma}\label{elements in M psi A} Let $M_{\psi_{A}}=\left\{m\in M \mid \psi^{m}_{A}(n)= \psi_{A}(n), \forall n\in N\right\}$. Then we have
\[M_{\psi_{A}}= \left \{\begin{bmatrix} a_{11} & a_{12} & a_{13} & & & \\ a_{21} & a_{22}  & a_{23} & & & \\ 0 & 0 & a & & & \\  & & & a &  y_{12} & y_{13} \\
& & & 0 & y_{22} & y_{23}\\
& & & 0 & y_{32} & y_{33} \end{bmatrix} \mid a \in F^{\times} \right \}.\]
\end{lemma}

\begin{proof} Let $g=\begin{bmatrix} g_1 & 0\\
0 & g_2 \end{bmatrix} \in M$. Then $ g \in M_{\psi_A}$ if and only if $ Ag_1=g_2A$. It follows that $g \in M_{\psi_A}$ if and only if $g_1= \begin{bmatrix}
a_{11} & a_{12} & a_{13}\\
a_{21} & a_{22} & a_{23} \\
0 & 0 & a \end{bmatrix}$ and $g_2=\begin{bmatrix}
a & y_{12} & y_{13}\\
0 & y_{22} & y_{23}\\
0 & y_{32} & y_{33} \end{bmatrix}$.
\end{proof}

\begin{lemma}
Let $Z=Z(G)$ be the center of $G$. Then, we have 
\[M_{\psi_{A}} \simeq Z \times H_{A}.\]
\end{lemma}

\begin{proof} Trivial. 
\end{proof}

\subsection{Character computation for $\rho$.} Let $\rho_{1}=\ind_{U_{1}}^{M_{1}}\mu_{1}$ and $\rho_{2}= \ind_{U_{2}}^{M_{2}}\mu_{2}$. In this section, we calculate the character of the representation \[\rho=\theta|_{F^{\times}} \otimes \ind_{U_{A}}^{H_{A}}\mu \simeq \theta|_{F^{\times}} \otimes (\rho_{1} \otimes \rho_{2}).\]

\subsubsection{Character computation of $\rho_1$} Let $\mu_{1}$ be same as above. Consider the representation \[\rho_{1}=  \ind_{U_{1}}^{M_{1}}\mu_{1} \] of $M_{1}$. Let $\chi_{\rho_{1}}$ be the character of $\rho_{1}$.
Let \[S_{1}= \left \{ \begin{bmatrix}
a & 0&0 \\
0& b &0 \\
0& 0& 1 \end{bmatrix} \mid a, b \in F^{\times} \right \}\]
 and \[ S_2= \left \{ \begin{bmatrix}
p & q & 0\\
r & 0 & 0\\
0 & 0& 1 \end{bmatrix} \mid  p \in F , q, r \in F^{\times} \right \}.\]
It is easy to see that $S=S_{1} \cup S_{2}$ is a set of left coset representatives of $U_{1}$ in $M_{1}$.

\begin{lemma}\label{lemma 1 in rho1 character calculation} Let \[m=\begin{bmatrix}
a_{11} & a_{12} & a_{13}\\
a_{21} & a_{22} & a_{23} \\
0 & 0 & 1 \end{bmatrix} \in M_{1}, t=\begin{bmatrix}
a & 0 & 0\\
0 & b & 0 \\
0 & 0 & 1  \end{bmatrix} \in S_{1}\]
Then, $t^{-1}mt \in U_{1}$ if and only if $a_{11}=a_{22}=1 ~\text{and}~ a_{21}=0.$ In particular, for $m\in M_{1}$ with $a_{11}=a_{22}=1$ and $a_{21}=0$, we have
\[ \sum_{t \in S_{1}}\mu_{1}(t^{-1}mt)=\sum_{a, b \in F^{\times}}\psi_0(a^{-1}ba_{12}+b^{-1}a_{23}).\]
\end{lemma}

\begin{proof}
For $m\in M_{1}$ and $t\in S_{1}$, we have
\[t^{-1}mt=\begin{bmatrix}
a_{11} & a^{-1}b a_{12} & a^{-1}a_{13}\\
b^{-1}a a_{21} & a_{22} & b^{-1}a_{23} \\
0 & 0 & 1 \end{bmatrix}. \]
Thus it follows that $t^{-1}mt\in U_{1}$ if and only if $a_{11}=a_{22}=1$ and $a_{21}=0$. Clearly, we have 
\[ \sum_{t \in S_{1}}\mu_{1}(t^{-1}mt)=\sum_{a, b \in F^{\times}}\psi_0(a^{-1}ba_{12}+b^{-1}a_{23}).\]
Hence the result. 
\end{proof}

\begin{lemma}\label{lemma 2 in rho1 character calculation}
Let \[m=\begin{bmatrix}
a_{11} & a_{12} & a_{13}\\
a_{21} & a_{22} & a_{23} \\
0 & 0 & 1 \end{bmatrix} \in M_{1}~\text{and}~ t=\begin{bmatrix} p & q & 0\\
r & 0 & 0\\
0 & 0 & 1 \end{bmatrix} \in S_2.\] Suppose that $a_{21}=0$. Then, \[t^{-1}mt \in U_{1} ~\text{if and only if}~ a_{11}=a_{22}=1~\text{and}~ a_{12}=0.\]
In particular, for $m\in M_{1}$ with $a_{11}=a_{22}=1$ and $a_{21}=a_{12}=0$, we have
\[ \sum_{t \in S_{2}}\mu_{1}(t^{-1}mt)=\sum_{\substack{p \in F \\ r, q \in F^{\times}}}\psi_{0}(-pq^{-1}r^{-1}a_{23}+q^{-1}a_{13}).\]
\end{lemma}

\begin{proof} Let $m\in M_{1}$ and $t\in S_{2}$. If $a_{21}=0$, we have, 
\[t^{-1}mt=\begin{bmatrix} a_{22} & & 0 & & r^{-1}a_{23}\\ 
pq^{-1}(a_{11}-a_{22})+rq^{-1}a_{12} & & a_{11} & & q^{-1}a_{13}-pq^{-1}r^{-1}a_{23}\\
0 & & 0 & & 1 \end{bmatrix}. \]
Then, $t^{-1}mt \in U_{1}$ if and only if $a_{22}=a_{11}=1$ and $a_{12}=0$. Clearly, we have
\[ \sum_{t \in S_{2}}\mu_{1}(t^{-1}mt)=\sum_{\substack{p \in F \\ r, q \in F^{\times}}}\psi_{0}(-pq^{-1}r^{-1}a_{23}+q^{-1}a_{13}).\]
Hence the result. 
\end{proof}

\begin{lemma}\label{lemma 3 in rho1 character calculation}
Let \[m=\begin{bmatrix}
a_{11} & a_{12} & a_{13}\\
a_{21} & a_{22} & a_{23} \\
0 & 0 & 1 \end{bmatrix} \in M_{1} ~\text{and}~ t=\begin{bmatrix}p & q & 0\\
r & 0 & 0\\
0 & 0 & 1 \end{bmatrix} \in S_2.\] 
Suppose that $a_{21} \neq 0$.
\begin{enumerate}
\item[a)] If $p=0$, then $t^{-1}mt \in U_{1}$ if and only if $a_{11}=a_{22}=1$ and $a_{12}=0$. In particular, we have 
\[\sum_{t\in S_{2}}\mu_{1}(t^{-1}mt)= \sum_{r,q\in F^{\times}}\psi_{0}(qr^{-1}a_{21}+ q^{-1}a_{13}). \]
\item[b)] If $p \neq 0$, then 
$t^{-1}mt \in U_{1}$ if and only if $a_{11}+a_{22}=2$, $a_{12}=\frac{-(a_{11}-1)^{2}}{a_{21}}$ and $r=\big (\frac{pa_{21}}{a_{11}-1}\big )$. In particular, we have 
\[\sum_{t \in S_{2}}\mu_{1}(t^{-1}mt)= -\sum_{\substack{q \in  F^{\times}}} \psi_{0}(q^{-1}(-\delta + a_{13})),\]
where $\delta= a_{21}^{-1}a_{23}(a_{11}-1)$.
\end{enumerate}
\end{lemma}

\begin{proof}
Let \[m=\begin{bmatrix}
a_{11} & a_{12} & a_{13}\\
a_{21} & a_{22} & a_{23} \\
0 & 0 & 1 \end{bmatrix} \in M_{1}\] and suppose that $a_{21} \neq 0$. In case a), since $p =0$, we have 
\[t^{-1}mt=\begin{bmatrix} a_{22} & qr^{-1}a_{21} & r^{-1}a_{23}\\
rq^{-1}a_{12} & a_{11} & q^{-1}a_{13}\\
0 & 0 & 1 \end{bmatrix}.\]
Thus it follows that $t^{-1}mt\in U_{1}$ if and only if $a_{22}=a_{11}=1$ and $a_{12}=0$. In particular, we have 
\[\sum_{t\in S_{2}}\mu_{1}(t^{-1}mt)= \sum_{r,q\in F^{\times}}\psi_{0}(qr^{-1}a_{21}+ q^{-1}a_{13}). \]
In case b), since $p \neq 0$, we have 
\[t^{-1}mt=\begin{bmatrix} a_{22}+pr^{-1}a_{21} & qr^{-1}a_{21} & r^{-1}a_{23}\\
pq^{-1}(a_{11}-a_{22})+rq^{-1}a_{12}-p^2q^{-1}r^{-1}a_{21} & a_{11}-pr^{-1}a_{21} & q^{-1}a_{13}-pq^{-1}r^{-1}a_{23}\\
0 & 0 & 1 \end{bmatrix}.\]
Clearly, we have $t^{-1}mt\in U_{1}$ if and only if
$a_{11}=1+pr^{-1}a_{21}$, $a_{22}=1-pr^{-1}a_{21}$ and $pq^{-1}(a_{11}-a_{22})+rq^{-1}a_{12}-p^2q^{-1}r^{-1}a_{21}=0$. \\

Using $a_{11}-a_{22}=2pr^{-1}a_{21}$, $a_{11}+a_{22}=2$ and $\det(t^{-1}mt)=1$, it follows that $a_{12}=\frac{-(a_{11}-1)^{2}}{a_{21}}$ and $r=\big (\frac{pa_{21}}{a_{11}-1}\big )$.\\

 In particular, taking $\delta= a_{21}^{-1}a_{23}(a_{11}-1)$ we have 
\begin{align*}
\sum_{t \in S_{2}}\mu_1(t^{-1}mt)&= -\sum_{\substack{q \in  F^{\times}}} \psi_0(q^{-1}(-{a_{21}}^{-1}a_{23}(a_{11}-1)+a_{13}))\\
&= -\sum_{\substack{q \in  F^{\times}}} \psi_{0}(q^{-1}(-\delta + a_{13})). 
\end{align*}
Hence the result.
\end{proof}

\noindent We summarize the character values of $\rho_{1}$ in the table below.

\begin{center}
\captionof{table}{ Character of $\rho_1$}
\label{Character of rho1}

\scriptsize
\resizebox{1.2\textwidth}{!}{
    \begin{tabular}{| c | c | c | c | c | c |}
    \hline
    & &    & & & \\
   Type of $m$ & $m$ & $\chi_{\rho_1}(m)$ &  Type of $m$ & $m$ & $\chi_{\rho_1}(m)$\\
   & & & & &\\
   \hline 
   
   & & & & & \\
 Type-1 & $\left \{\begin{bmatrix} 1 & a_{12} & a_{13}  \\ 0 & 1 & 0\\
0 & 0 & 1 \end{bmatrix} \,\middle \vert\, a_{12} \in F^{\times} \right \}$ & $(1-q)$ & Type-6 & $\left \{\begin{bmatrix} 1 & 0 & a_{13}  \\ 0 & 1 & a_{23}\\
0 & 0 & 1 \end{bmatrix} \,\middle \vert\, a_{13},a_{23} \in F^{\times} \right \} $  & $(1-q)$\\
& & & & &\\
\hline
& & & & &\\
Type-2 & $\left \{ \begin{bmatrix} 1 & a_{12} & a_{13}\\
         0 & 1 & a_{23}\\
         0 & 0 & 1
         \end{bmatrix} \,\middle \vert\, a_{12},a_{23} \in F^{\times}\right \} $  & $1$ & Type-7  & $\left \{\begin{bmatrix} 1 & 0 & 0\\
         a_{21} & 1 & a_{23}\\
         0 & 0 & 1
         \end{bmatrix} \,\middle \vert\, a_{21} \in F^{\times} \right\} $ & $(1-q)$ \\
& & & & & \\
\hline
& & & & &\\
         Type-3 & $ \left \{ \begin{bmatrix} 1 & 0 & 0\\
         0 & 1 & a_{23}\\
         0 & 0 & 1
         \end{bmatrix} \,\middle \vert\, a_{23} \in F^{\times} \right \} $ & $(1-q)$ & Type-8 &$\left \{\begin{bmatrix} 1 & 0 & a_{13}\\
         a_{21} & 1 & a_{23}\\
         0 & 0 & 1
         \end{bmatrix} \,\middle \vert\, a_{21},a_{13} \in F^{\times} \right \}$ & $1$\\
         & & & & & \\
         \hline 
         & & & & &\\
         Type-4 & $ \left \{\begin{bmatrix} 1 & 0 & 0\\
         0 & 1 & 0\\
         0 & 0 & 1
         \end{bmatrix} \right \} $ & $(1-q)(1-q^2)$ & Type-9 & $\left \{\begin{bmatrix} a_{11} & a_{12} & a_{13}\\
         a_{21} & a_{22} & a_{23}\\
         0 & 0 & 1
         \end{bmatrix} \,\middle\vert\, \stackrel{a_{21} \in F^{\times},\, a_{13}= \delta}{a_{12}=-{a_{21}}^{-1}{(a_{11}-1)}^2}  \right \}$ & $(1-q)$\\
           &   & & & & \\
           \hline 
         & & & & &\\
        Type-5 & $\left \{ \begin{bmatrix} 1 & 0 & a_{13} \\ 0 & 1 & 0\\
0 & 0 & 1 \end{bmatrix} \,\middle \vert\, a_{13} \in F^{\times} \right \}$ & $(1-q)$ & Type-10 & $\left \{\begin{bmatrix} a_{11} & a_{12} & a_{13}\\
         a_{21} & a_{22} & a_{23}\\
         0 & 0 & 1
         \end{bmatrix} \,\middle \vert\, \stackrel{a_{21} \in F^{\times},\, a_{13} \neq \delta }{a_{12}=-{a_{21}}^{-1}{(a_{11}-1)}^2} \right \}$ & $1$\\
           &   & & & & \\
           \hline

    \end{tabular} 
    }
\end{center}
If $m\in M_{1}$ is not one of the types mentioned in Table~\ref{Character of rho1}, then $\chi_{\rho_1}(m)=0$.

\subsubsection{Character computation of $\rho_2$}
Let $\mu_{2}$ be same as above. Consider the representation 
\[ \rho_{2}=\ind_{U_{2}}^{M_{2}}\mu_{2}\]
of $M_{2}.$ Let $\chi_{\rho_{2}}$ be the character of $\rho_{2}$. 
Let \[S_{3}= \left \{ \begin{bmatrix}
1& 0 &0 \\
0& a &0 \\
0& 0 & b \end{bmatrix} \mid a, b \in F^{\times} \right \}\]
 and \[ S_{4} = \left \{ \begin{bmatrix}
1 & 0 & 0\\
0 & p & q\\
0 & r & 0 \end{bmatrix} \mid  p \in F , q, r \in F^{\times} \right \}.\]
It is easy to see that $S=S_ {3}  \cup S_{4}$ is a set of left coset representatives of $U_{2}$ in $M_{2}$. 

\begin{lemma} Let \[m=\begin{bmatrix}
1 & y_{12} & y_{13}\\
0 & y_{22} & y_{23} \\
0 & y_{32} & y_{33} \end{bmatrix} \in M_{2}, t=\begin{bmatrix}
1 & 0 & 0\\
0 & a & 0 \\
0 & 0 & b  \end{bmatrix} \in S_{3}.\]
Then, $t^{-1}mt \in U_{2}$ if and only if $y_{22}=y_{33}=1$ and $y_{32}=0$. In particular, for $m\in M_{2}$ with $y_{22}=y_{33}=1$ and $y_{32}=0$, we have
\[ \sum_{t \in S_{3}}\mu_{2}(t^{-1}mt)=\sum_{a, b \in F^{\times}}\psi_0(ay_{12}+ ba^{-1}y_{23}).\]
\end{lemma}
\begin{proof} Similar to Lemma~\ref{lemma 1 in rho1 character calculation}. 
\end{proof}

\begin{lemma}
Let \[m=\begin{bmatrix}
1 & y_{12} & y_{13}\\
0 & y_{22} & y_{23} \\
0 & y_{32} & y_{33} \end{bmatrix} \in M_{2}~ \text{and}~ t=\begin{bmatrix}1 & 0 & 0\\
0 & p & q\\
0 & r & 0 \end{bmatrix} \in S_{4}.\] Suppose that $y_{32}=0$. Then, \[t^{-1}mt \in U_{2}~ \text{if and only if}~ y_{22}=y_{33}=1~ \text{and}~ y_{23}=0.\]
In particular, for $m\in M_{2}$ with $y_{22}=y_{33}=1$ and $y_{32}=y_{23}=0$, we have 
\[ \sum_{t \in S_{4}}\mu_{2}(t^{-1}mt)=\sum_{\substack{p \in F \\ r, q \in F^{\times}}}\psi_{0}(py_{12}+ry_{13}).\]
\end{lemma}

\begin{proof} Similar to Lemma~\ref{lemma 2 in rho1 character calculation}. 
\end{proof}

\begin{lemma}
Let \[m=\begin{bmatrix}
1 & y_{12} & y_{13}\\
0 & y_{22} & y_{23} \\
0 & y_{32} & y_{33} \end{bmatrix} \in M_{2}~\text{and}~ t=\begin{bmatrix}1 & 0 & 0\\
0 & p & q\\
0 & r & 0 \end{bmatrix} \in S_{4}. \]
Suppose that $y_{32} \neq 0$.
\begin{enumerate}
\item[a)] If $p=0$, then $t^{-1}mt \in U_{2}$ if and only if $y_{22}=y_{33}=1$ and $y_{23}=0$. In particular, we have 
\[\sum_{t\in S_{4}}\mu_{2}(t^{-1}mt)= \sum_{r,q\in F^{\times}}\psi_{0}(qr^{-1}y_{32}+ ry_{13}). \]
\item[b)] If $p \neq 0$, then $t^{-1}mt \in U_{2}$ if and only if $y_{22}+y_{33}=2$, $y_{23}=-\frac{(y_{22}-1)^2}{{y_{32}}}$ and $r={\left(\frac{y_{32}p}{y_{22}-1}\right)}$. In particular, we have 
\[\sum_{t \in S_{4}}\mu_{2}(t^{-1}mt)= -\sum_{\substack{p \in  F^{\times}}} \psi_{0}(p(\gamma + y_{12})),\]
where $\gamma=y_{32}y_{13}{(y_{22}-1)}^{-1}$.
\end{enumerate}
\end{lemma}

\begin{proof} Similar to Lemma~\ref{lemma 3 in rho1 character calculation}. 
\end{proof}
We record the character values of $\rho_{2}$ in the following table.

\begin{center}
\captionof{table}{ Character of $\rho_2$}
\label{Character of rho2}

\scriptsize
\resizebox{1.2\textwidth}{!}{
    \begin{tabular}{| c | c | c | c | c | c |}
    \hline
    & &    & & & \\
   Type of $m$ & $m$ & $\chi_{\rho_2}(m)$ &  Type of $m$ & $m$ & $\chi_{\rho_2}(m)$\\
   & & & & &\\
   \hline 
   
   & & & & & \\
 Type-1 & $\left \{\begin{bmatrix} 1 & 0 & y_{13}  \\ 0 & 1 & y_{23}\\
0 & 0 & 1 \end{bmatrix} \,\middle \vert\, y_{23} \in F^{\times} \right \}$ & $(1-q)$ & Type-6 & $\left \{\begin{bmatrix} 1 & y_{12} & y_{13}  \\ 0 & 1 & 0\\
0 & 0 & 1 \end{bmatrix} \,\middle \vert\, y_{12},y_{13} \in F^{\times} \right \} $  & $(1-q)$\\
& & & & &\\
\hline
& & & & &\\
Type-2 & $\left \{ \begin{bmatrix} 1 & y_{12} & y_{13}\\
         0 & 1 & y_{23}\\
         0 & 0 & 1
         \end{bmatrix} \,\middle \vert\, y_{12},y_{23} \in F^{\times}\right \} $  & $1$ & Type-7  & $\left \{\begin{bmatrix} 1 & y_{12} & 0\\
         0 & 1 & 0\\
         0 & y_{32} & 1
         \end{bmatrix} \,\middle \vert\, y_{32} \in F^{\times} \right\} $ & $(1-q)$ \\
& & & & & \\
\hline
& & & & &\\
         Type-3 & $ \left \{ \begin{bmatrix} 1 & y_{12} & 0\\
         0 & 1 & 0\\
         0 & 0 & 1
         \end{bmatrix} \,\middle \vert\, y_{12} \in F^{\times} \right \} $ & $(1-q)$ &  Type-8 & $\left \{\begin{bmatrix} 1 & y_{12} & y_{13}\\
         0 & 1 & 0\\
         0 & y_{32} & 1
         \end{bmatrix} \,\middle \vert\, y_{13},y_{32} \in F^{\times} \right \}$ & $1$\\
         & & & & & \\
         \hline 
         & & & & &\\
         Type-4 & $ \left \{\begin{bmatrix} 1 & 0 & 0\\
         0 & 1 & 0\\
         0 & 0 & 1
         \end{bmatrix} \right \} $ & $(1-q)(1-q^2)$ & Type-9 & $\left \{\begin{bmatrix} 1 & y_{12} & y_{13}\\
         0 & y_{22} & y_{23}\\
         0 & y_{32} & y_{33}
         \end{bmatrix} \,\middle \vert\, \stackrel{ y_{32} \in F^{\times},\, y_{12}=-\gamma,}{y_{23}=-{y_{32}}^{-1}{(y_{22}-1)}^2} \right \}$ & $(1-q)$ \\
           &   & & & & \\
           \hline
            & & & & &\\
        Type-5 & $\left \{ \begin{bmatrix} 1 & 0 & y_{13} \\ 0 & 1 & 0\\
0 & 0 & 1 \end{bmatrix} \,\middle \vert\, y_{13} \in F^{\times} \right \} $ & $(1-q)$ & Type-10 & $\left \{\begin{bmatrix} 1 & y_{12} & y_{13}\\
         0 & y_{22} & y_{23}\\
         0 & y_{32} & y_{33}
         \end{bmatrix} \,\middle \vert\, \stackrel{ y_{32} \in F^{\times}, \, y_{12} \neq -\gamma,}{y_{23}=-{y_{32}}^{-1}{(y_{22}-1)}^2} \right \}$ & $1$\\
           &   & & & & \\
           \hline
           
    \end{tabular} 
    }
\end{center}

\vspace{0.2 cm}
\noindent If $m\in M_{2}$ is not one of the types mentioned above, we have $\chi_{\rho_2}(m)=0$.\\

For $1\leq i,j\leq 10$, we let \[T(i,j)=\{k=(m_{1}, m_{2})\in H_{A} \mid m_{1}\in \Type-i, m_{2}\in \Type-j\}. \] 

\begin{theorem}
Let $\rho=\theta|_{F^{\times}} \otimes \rho_1 \otimes \rho_2$. 
Let $\chi_{\rho}$ be the character of $\rho$. For $m=(a,m_1,m_2) \in Z \times M_1 \times M_2$, we have 
\[\chi_{\rho}(m)=\theta(a)\chi_{\rho_1}(m_1)\chi_{\rho_2}(m_2)\]
where $(m_{1}, m_{2}) \in T(i,j)$, $i,j \in \{ 1 ,\hdots, 10\}.$ Otherwise, $\chi_{\rho}(m)=0$.
\end{theorem}

\begin{proof}
We summarize the results from Table \eqref{Character of rho1} and Table \eqref{Character of rho2} below.

\begin{landscape}
\captionof{table}{ Character of $\rho$}
\label{Character of thetaxrho1xrho2 }

\scriptsize
\resizebox{2\textwidth}{!}{
    \begin{tabular}{| c | c | c | c | c | c | c | c | c | c | c | }
    \hline
    & &    & & & & & & & & \\
    & Type-1 & Type-2 & Type-3 & Type-4 & Type-5 & Type-6 & Type-7 & Type-8 & Type-9 & Type-10\\
   \hline 
   
   & & & & & & & & & &\\
   Type-1& $\theta(a)(1-q)^2$ & $\theta(a)(1-q)$ & $\theta(a)(1-q)^2$ & $\theta(a)(1-q)^2(1-q^2)$ & $\theta(a)(1-q)^2$ & $\theta(a)(1-q)^2$ & $\theta(a)(1-q)^2$ & $\theta(a)(1-q)$ & $\theta(a)(1-q)^2$ & $\theta(a)(1-q)$\\
      & & & & & & & & & &\\
    \hline
    
       & & & & & & & & & &\\
       Type-2 & $\theta(a)(1-q)$ & $\theta(a)$ & $\theta(a)(1-q)$ & $\theta(a)(1-q)(1-q^2)$ & $\theta(a)(1-q)$ & $\theta(a)(1-q)$ & $\theta(a)(1-q)$ & $\theta(a)$ & $\theta(a)(1-q)$ & $\theta(a)$\\
          & & & & & & & & & &\\
\hline

   & & & & & & & & & &\\
    Type-3& $\theta(a)(1-q)^2$ & $\theta(a)(1-q)$ & $\theta(a)(1-q)^2$ & $\theta(a)(1-q)^2(1-q^2)$ & $\theta(a)(1-q)^2$ & $\theta(a)(1-q)^2$ & $\theta(a)(1-q)^2$ & $\theta(a)(1-q)$ & $\theta(a)(1-q)^2$ & $\theta(a)(1-q)$ \\
       & & & & & & & & & &\\
\hline

   & & & & & & & & & &\\
Type-4& $\theta(a)(1-q)^2(1-q^2)$ & $\theta(a)(1-q)(1-q^2)$ & $\theta(a)(1-q)^2(1-q^2)$ & $\theta(a)(1-q)^2(1-q^2)^2$ & $\theta(a)(1-q)^2(1-q^2)$ & $\theta(a)(1-q)^2(1-q^2)$ & $\theta(a)(1-q)^2(1-q^2)$ & $\theta(a)(1-q)(1-q^2)$ & $\theta(a)(1-q)^2(1-q^2)$ & $\theta(a)(1-q)(1-q^2)$\\
   & & & & & & & & & &\\
\hline

& & & & & & & & & &\\
   Type-5& $\theta(a)(1-q)^2$ & $\theta(a)(1-q)$ & $\theta(a)(1-q)^2$ & $\theta(a)(1-q)^2(1-q^2)$ & $\theta(a)(1-q)^2$ & $\theta(a)(1-q)^2$ & $\theta(a)(1-q)^2$ & $\theta(a)(1-q)$ & $\theta(a)(1-q)^2$ & $\theta(a)(1-q)$\\
      & & & & & & & & & &\\
    \hline

& & & & & & & & & &\\
   Type-6 & $\theta(a)(1-q)^2$ & $\theta(a)(1-q)$ & $\theta(a)(1-q)^2$ & $\theta(a)(1-q)^2(1-q^2)$ & $\theta(a)(1-q)^2$ & $\theta(a)(1-q)^2$ & $\theta(a)(1-q)^2$ & $\theta(a)(1-q)$ & $\theta(a)(1-q)^2$ & $\theta(a)(1-q)$\\
      & & & & & & & & & &\\
    \hline
    
    & & & & & & & & & &\\
   Type-7& $\theta(a)(1-q)^2$ & $\theta(a)(1-q)$ & $\theta(a)(1-q)^2$ & $\theta(a)(1-q)^2(1-q^2)$ & $\theta(a)(1-q)^2$ & $\theta(a)(1-q)^2$ & $\theta(a)(1-q)^2$ & $\theta(a)(1-q)$ & $\theta(a)(1-q)^2$ & $\theta(a)(1-q)$\\
      & & & & & & & & & &\\
    \hline
       
       & & & & & & & & & &\\
       Type-8 & $\theta(a)(1-q)$ & $\theta(a)$ & $\theta(a)(1-q)$ & $\theta(a)(1-q)(1-q^2)$ & $\theta(a)(1-q)$ & $\theta(a)(1-q)$ & $\theta(a)(1-q)$ & $\theta(a)$ & $\theta(a)(1-q)$ & $\theta(a)$\\
          & & & & & & & & & &\\
\hline

 & & & & & & & & & &\\
   Type-9& $\theta(a)(1-q)^2$ & $\theta(a)(1-q)$ & $\theta(a)(1-q)^2$ & $\theta(a)(1-q)^2(1-q^2)$ & $\theta(a)(1-q)^2$ & $\theta(a)(1-q)^2$ & $\theta(a)(1-q)^2$ & $\theta(a)(1-q)$ & $\theta(a)(1-q)^2$ & $\theta(a)(1-q)$\\
      & & & & & & & & & &\\
    \hline
    
 & & & & & & & & & &\\
       Type-10 & $\theta(a)(1-q)$ & $\theta(a)$ & $\theta(a)(1-q)$ & $\theta(a)(1-q)(1-q^2)$ & $\theta(a)(1-q)$ & $\theta(a)(1-q)$ & $\theta(a)(1-q)$ & $\theta(a)$ & $\theta(a)(1-q)$ & $\theta(a)$\\
          & & & & & & & & & &\\
\hline
    \end{tabular} 
    }
    \end{landscape}
\end{proof}

\section{Character calculation for $\pi_{N,\psi_{A}}$} \hfill

\begin{lemma} Let $m=ah \in M_{\psi_{A}}$, where $a\in Z$ and $h\in H_{A}$. Then,
\[\Theta_{N,\psi_A}(m)=\theta(a)\Theta_{N,\psi_A}(h).\]
\end{lemma}

\begin{proof} We have
\begin{align*}
\Theta_{N,\psi_{A}}(m) &=\Theta_{N,\psi_{A}}(ah)\\
&= \frac{1}{|N|}\sum_{n \in N}\Theta_{\theta}(ahn)\overline{\psi_{A}(n)}\\
&=\frac{1}{|N|}\sum_{n \in N}\tr(\pi(ahn))\overline{\psi_A(n)}\\
&=\frac{1}{|N|}\sum_{n \in N}\tr(\pi(a)\pi(hn))\overline{\psi_A(n)}\\
&=\omega_{\pi}(a)\frac{1}{|N|}\sum_{n \in N}\tr(\pi(hn))\overline{\psi_A(n)}\\
&=\omega_{\pi}(a)\Theta_{N,\psi_A}(h)
\end{align*}
where $\omega_{\pi}$ is the central character of $\pi$. Explicitly, we have \[\Theta_{\theta}(a)=\tr(\pi(a))=\tr(\omega_{\pi}(a))=\omega_{\pi}(a)\dim(\pi).\]
Using Theorem \ref{character value of cuspidal representation (Dipendra)}, it is easy to see that $$\Theta_{\theta}(a)=\theta(a)\dim(\pi).$$
Thus, we have $\omega_{\pi}(a)=\theta(a)$ and the result follows. 
\end{proof}

\begin{lemma}
Let $\tau=\begin{bmatrix} 0 & w_0\\
w_0 & 0 \end{bmatrix}$. For $1 \leq i,j \leq 10$, we have 
\[T(j,i)=\tau{T(i,j)}^{\top}{\tau}^{-1}.\]
\end{lemma}

\begin{proof} Trivial.
\end{proof}

\begin{theorem} \label{T(i,j)=T(j,i)}
 Let $m' \in T(j,i)$. Then there exists $m \in T(i,j)$ such that 
\[\Theta_{N,\psi_A}(m)=\Theta_{N,\psi_A}(m').\]
\end{theorem}

\begin{proof}
Let $m'= \begin{bmatrix} m_1' & 0\\
0 & m_2' \end{bmatrix} \in T(j,i)$. Since $T(j,i)= \tau T(i,j)^{\top}\tau^{-1}$, it follows that, 
\[m'=\tau m^{\top} \tau^{-1}\]
for some $m=\begin{bmatrix} m_1 & 0  \\
0 & m_2 \end{bmatrix}\in T(i,j)$. Thus we have $m_1'=w_0{m_2}^{\top}{w_0}^{-1}$ and $m_2'=w_0{m_1}^{\top}{w_0}^{-1}$. 
Since $m_{1}'\in M_{1}$, clearly $\psi_A(X)=\psi_A(({w_0{m_2}^{\top}{w_0}})^{-1}X)$. We have 
\begin{align*}
\Theta_{N,\psi_A}(m')&= \frac{1}{|N|}\sum_{X \in \M(3,F)} \Theta_{\theta} \begin{bmatrix} w_0{m_2}^{\top}{w_0}^{-1} & X \\
0 & w_0{m_1}^{\top}{w_0}^{-1} \end{bmatrix} \overline{\psi_A(X)}\\
&= \frac{1}{|N|}\sum_{X \in \M(3,F)} \Theta_{\theta}\left( \begin{bmatrix} w_0 & 0 \\
0 & w_0 \end{bmatrix} \begin{bmatrix} {m_2}^{\top} & {w_0}^{-1}Xw_0 \\
0 & {m_1}^{\top} \end{bmatrix} \begin{bmatrix} w_0 & 0 \\
0 & w_0 \end{bmatrix} \right) \overline{\psi_A(X)}\\
&= \frac{1}{|N|}\sum_{X \in \M(3,F)} \Theta_{\theta}\left( \begin{bmatrix} {m_2}^{\top} & {w_0}^{-1}Xw_0 \\
0 & {m_1}^{\top} \end{bmatrix} \right) \overline{\psi_A(X)}.
\end{align*}
On the other hand, using $\tr(A({w_0}^{-1}X^{\top}w_0))=\tr(AX)$ we have 
\begin{align*}
\Theta_{N,\psi_A}(m) &= \frac{1}{|N|}\sum_{X \in \M(3,F)} \Theta_{\theta} \begin{bmatrix} m_1 & X \\
0 & m_2 \end{bmatrix} \overline{\psi_A(X)}\\
&= \frac{1}{|N|}\sum_{X \in \M(3,F)} \Theta_{\theta} \begin{bmatrix} m_1 & {X}^{\top} \\
0 & m_2 \end{bmatrix} \overline{\psi_A(X^{\top})}\\
&=\frac{1}{|N|} \sum_{X \in \M(3,F)} \Theta_{\theta} \begin{bmatrix} m_1 & {w_0}^{-1}{X}^{\top}w_0 \\
0 & m_2 \end{bmatrix} \overline{\psi_A({w_0}^{-1}X^{\top}w_0)}. \\
&= \frac{1}{|N|}\sum_{X \in \M(3,F)} \Theta_{\theta} \begin{bmatrix} m_1 & {w_0}^{-1}{X}^{\top}w_0 \\
0 & m_2 \end{bmatrix} \overline{\psi_A(X)}.
\end{align*}
Since \begin{align*}
\Rank \left (\begin{bmatrix} {m_2}^{\top}-1 & w_0^{-1}Xw_0 \\
0 & {m_1}^{\top}-1 \end{bmatrix} \right)
&=\Rank \left (\begin{bmatrix} {m_2}-1 & 0 \\
w_0^{-1}{X}^{\top}w_0 & {m_1}-1 \end{bmatrix} \right)\\
&=\Rank \left ( \begin{bmatrix} 0 & 1 \\
1 & 0 \end{bmatrix} \begin{bmatrix} {m_2}-1 & 0 \\
w_0^{-1}{X}^{\top}w_0 & {m_1}-1 \end{bmatrix} \begin{bmatrix} 0& 1\\
1 & 0 \end{bmatrix}\right)\\
&=\Rank \left ( \begin{bmatrix} {m_1}-1 & w_0^{-1}{X}^{\top}w_0 \\
0 & {m_2}-1 \end{bmatrix} \right)
\end{align*}
we have, \[\dim(\ker\left (\begin{bmatrix} {m_2}^{\top}-1 & w_0^{-1}Xw_0 \\
0 & {m_1}^{\top}-1 \end{bmatrix} \right)=\dim(\ker  \left ( \begin{bmatrix} {m_1}-1 & w_0^{-1}{X}^{\top}w_0 \\
0 & {m_2}-1 \end{bmatrix} \right).\]
Hence, \[\Theta_{\theta}\begin{bmatrix} {m_2}^{\top} & w_0^{-1}Xw_0 \\
0 & {m_1}^{\top} \end{bmatrix} =\Theta_{\theta}\begin{bmatrix} {m_1} & w_0^{-1}{X}^{\top}w_0 \\
0 & {m_2} \end{bmatrix} \]
and the result follows. 
\end{proof}

\begin{remark} We have used the fact that $\Rank(M)=\Rank(M^{\top})$ and $\Rank(NMP)=\Rank(M)$ if $N$ and $P$ are invertible matrices. \\
\end{remark}

Let $m=(m_{1}, m_{2}) \in M_{1}\times M_{2}= H_{A}$. Suppose also that $m_{1}, m_{2}$ are unipotent. To calculate $\Theta_{N,\psi_A}(m)$, we need to compute $\Theta_{\theta}(h)$, where $h=\begin{bmatrix} m_1 & X\\
0 & m_2 \end{bmatrix}$. Using Theorem \ref{character value of cuspidal representation (Dipendra)}, it suffices to compute $\dim\Ker(h-1)$. We note that the following proposition is valid even when $H_{A}$ is a subgroup of $\GL(2n,F)$.

\begin{proposition} \label{dim(ker(h-1))}
Let $h=\begin{bmatrix}
m_1 & X \\
0 & m_2 \end{bmatrix}\in \GL(2n,F)$, where $(m_{1}, m_{2})\in H_{A}$. Suppose that $m_{1}$ and $m_{2}$ are also unipotent. Let $W'=\Ker(m_2-1)$.  Then, we have
\[\dim\Ker(h-1)=\dim \Ker(m_1-1)+\dim \Ker(m_2-1)-\dim(XW')+\dim\{XW' \cap \Image(m_1-1)\}.\]
\end{proposition}

\begin{proof}
Let $V$ be an $n$-dimensional vector space over $F$ and let $m_{1}, m_{2}, X$ be linear operators on $V$. Suppose that $\{e_1,\hdots, e_m\}$ is a basis for $\Ker(m_{1}-1)$ and $\{f_1,\hdots,f_k\}$ is a basis for $\Ker(m_{2}-1)$. Extending the basis of $\Ker(m_1-1)$ and $\Ker(m_{2}-1)$ we get ordered bases $\beta=\{e_1,\hdots,e_n\}$ and $\beta'=\{f_1,\hdots,f_n\}$ of $V$. Consider the ordered basis $\tilde{\beta}=\{(e_1,0),\hdots,(e_n,0),(0,f_1),\hdots,(0,f_n)\}$ of $V \oplus V$. We let $h$ to be the linear operator on $V \oplus V$ defined as follows. For $1 \leq i,j \leq n$, 
\[h((e_i,0))=(m_1,0)(e_i,0)=(m_1(e_i),0)\] and \[ h((0,f_j))=(X,m_2)(0,f_j)=(X(f_j),m_2(f_j)).\]
Then, 
\[[h]_{\tilde{\beta}}=\begin{bmatrix}
[m_1]_{\beta} & [X]_{\beta'} \\
0 & [m_2]_{\beta'} \end{bmatrix}\]
where \[[X]_{\beta'}=[X_1 ~ X_2~ \cdots~ X_n].\] 
Let \[W_1=\Span\{(m_1-1,0)(e_{m+1},0),\hdots,(m_1-1,0)(e_n,0)\}=\Image(m_1-1),\]
\[W_2=\Span\{(X,m_2-1)(0,f_1),\hdots,(X,m_2-1)(0,f_k)\}=XW'\]
and 
\[W_3=\Span\{(Xf_{k+1},(m_2-1)f_{k+1}),\hdots,(Xf_n,(m_2-1)f_n).\]
Clearly, 
\[\Image(h-1)=W_1 + W_2 + W_3.\]
It is easy to see that \[W_2 \cap W_3=\{0\}=W_1 \cap W_3.\]
Since $\dim(\Ker(m_2-1))=k$, we have that \[\dim(W_3)=\dim(\Image(m_2-1)).\] Therefore, \[\dim(\Image(h-1))=\dim(\Image(m_1-1))+\dim(\Image(m_2-1)) +\dim(XW')-\dim(XW' \cap \Image(m_1-1)).\]
Hence the result.
\end{proof}

\begin{remark} \label{X}
Let $h$ be as in Proposition \ref{dim(ker(h-1))}. We note that\[XW'=\Span\{X_1,X_2,\hdots,X_k\}.\] We will continue to use this in our character calculations at several instances to follow. 
\end{remark}

Let $m=\begin{bmatrix} m_{1} & 0 \\ 0 & m_{2}\end{bmatrix}\in H_{A}$, where $m_{1}\in M_{1}$, $m_{2}\in M_{2}$. Throughout we write $W^{'}=\Ker(m_{2}-1)$. For $X\in \M(3,F)$, we let $h=\begin{bmatrix} m_{1} & X \\ 0 & m_{2}\end{bmatrix}$. For $\beta\in F$, we define 
\[S(\beta)=\{ X \in M(3,F) \mid \tr(A{m_1}^{-1}X)=\tr(AX)=\beta\}.\]
Let \[E= \bigcup_{\substack{i \leq j\\ i,j \in \{1,2,4\}}}T{(i,j)}.\] 
We call $E$ to be the fundamental set. To determine $\Theta_{N,\psi_A}(m)$ for $m \in T{(i,j)}$, it is enough to compute $\Theta_{N,\psi_A}(m)$ for $m \in E$.

\begin{theorem} \label{Type(1,1)}
Let $m \in T(1,1)$. Then, we have \[\Theta_{N,\psi_A}(m)=(1-q)^2.\]
\end{theorem}

\begin{proof}
 We have \[ \Theta_{N,\psi_A}(m)=\frac{1}{|N|} \sum_{ X \in M(3,F)}\Theta_{\theta} \begin{bmatrix} m_{1} & X\\
 0 & m_{2}\end{bmatrix}\overline{\psi_{A}({m_{1}}^{-1}X}).\]
Note that $\dim\Ker(m_1-1)=2$, $\Image(m_1-1)=\Span \left \{ e_1 \right \}$ and $\Ker(m_2-1)=\Span \left \{ e_1,e_2 \right \}$.
To calculate the character value, we write 
 \[ \Theta_{N,\psi_{A}}(m)=\frac{1}{q^{9}} (A_{1}+A_{2})\]
 where 
 \[ A_{1}= \sum_{X \in S(0)}\Theta_{\theta} \begin{bmatrix} m_{1} & X\\
 0 & m_{2}\end{bmatrix}\overline{\psi_{A}({m_{1}}^{-1}X})\]
  and \[ A_{2}=\sum_{\beta \in F^{\times}}\sum_{X \in S(\beta)}\Theta_{\theta} \begin{bmatrix} m_{1} & X\\
 0 & m_{2}\end{bmatrix}\overline{\psi_A({m_1}^{-1}X}).\]
For simplicity,
we let $t=\dim(\Ker(h-1))$. To compute $A_{1}$, we find a partition of $S(0)$ according to the value of $t$ and compute the respective cardinalities. We record the details in the table below. 
 
\begin{center}
\captionof{table}{$A_1$}
\label{A_1}
\scriptsize
\scriptsize
\resizebox{1\textwidth}{!}{
    \begin{tabular}{| c | c | c | c | c | c |}
    \hline
  &  & &    &  &  \\
  & Partition of $S(0)$ & $\dim(XW')$ & $\dim(XW' \cap \Image(m_1-1))$ & $t=\dim(\Ker(h-1))$ & Cardinality\\
  & & & &  & \\
   \hline 
   
 & & & & &\\
$1) (a)$ & $\left \{\begin{bmatrix} 0 & 0 & e  \\ 0 & 0 & f\\
0 & 0 & l \end{bmatrix} \right \}$ & $0$ & $0$ & $4$ & $q^3$\\
 & & & & & \\
\hline
& & & & & \\
$1)(b)$ & $\left \{ \begin{bmatrix} 0 & 0 & e\\
         0 & 0 & f\\
         0 & k & l         \end{bmatrix} \,\middle \vert\, k \in F^{\times} \right \} $  & $1$ & $0$ & $3$ & $(q-1)q^3$\\
 & & & & &\\
\hline
& & & & & \\
$2)(a)$ & $\left \{ \begin{bmatrix} a & b & e\\
         c & d & f\\
         0 & 0 & l         \end{bmatrix} \,\middle \vert\, ad-bc \neq 0 \right \} $  & $2$ & $1$ & $3$ & $(q^2-1)(q-1)q^4$\\
 & & & & &\\
\hline
& & & & & \\
$2)(b)$ & $\left \{ \begin{bmatrix} a & b & e\\
         c & d & f\\
         0 & k & l         \end{bmatrix} \,\middle \vert\, c,k \in F^{\times}, ad-bc \neq 0 \right \} $  & $2$ & $0$ & $2$ & $(q-1)^3q^5$\\
 & & & & &\\
\hline
& & & & & \\
$2)(c)$ & $\left \{ \begin{bmatrix} a & b & e\\
         0 & d & f\\
         0 & k & l         \end{bmatrix} \,\middle \vert\, k \in F^{\times}, ad \neq 0 \right \} $  & $2$ & $1$ & $3$ & $(q-1)^3q^4$\\
 & & & & &\\
\hline
& & & & & \\
$3)(a)$ & $\left \{ \begin{bmatrix} a & \gamma a & e\\
         c & \gamma c & f\\
         0 & 0 & l         \end{bmatrix} \,\middle \vert\, c \in F^{\times},\gamma \in F \right \} $  & $1$ & $0$ & $3$ & $(q-1)q^5$\\
 & & & & &\\
\hline
& & & & & \\
$3)(b)$ & $\left \{ \begin{bmatrix} a & \gamma a & e\\
         0 & 0 & f\\
         0 & 0 & l         \end{bmatrix} \,\middle \vert\, a \in F^{\times},\gamma \in F \right \} $  & $1$ & $1$ & $4$ & $(q-1)q^4$\\
 & & & & &\\
\hline
& & & & & \\
$3)(c)$ & $\left \{ \begin{bmatrix} a & \gamma a & e\\
         c & \gamma c & f\\
         0 & k & l         \end{bmatrix} \,\middle \vert\, c,k \in F^{\times},\gamma \in F \right \} $  & $2$ & $0$ & $2$ & $(q-1)^2q^5$\\
 & & & & &\\
\hline
& & & & & \\
$3)(d)$ & $\left \{ \begin{bmatrix} a & \gamma a & e\\
          0 & 0 & f\\
         0 & k & l         \end{bmatrix} \,\middle \vert\, k \in F^{\times},\gamma \in F \right \} $  & $2$ & $1$ & $3$ & $(q-1)^2q^4$\\
 & & & & &\\
\hline
& & & & & \\
$4)(a)$ & $\left \{ \begin{bmatrix} 0 & b & e\\
         0 & d & f\\
         0 & 0 & l         \end{bmatrix} \,\middle \vert\, d \in F^{\times} \right \} $  & $1$ & $0$ & $3$ & $(q-1)q^4$\\
 & & & & &\\
\hline
& & & & &\\
$4)(b)$ & $\left \{ \begin{bmatrix} 0 & b & e\\
         0 & 0 & f\\
         0 & 0 & l         \end{bmatrix} \,\middle \vert\, b \in F^{\times} \right \} $  & $1$ & $1$ & $4$ & $(q-1)q^3$\\
 & & & & &\\
\hline
& & & & &\\
$4)(c)$ & $\left \{ \begin{bmatrix} 0 & b & e\\
         0 & d & f\\
         0 & k & l         \end{bmatrix} \,\middle \vert\, k \in F^{\times} \right \} $  & $1$ & $0$ & $3$ & $(q^2-1)(q-1)q^3$\\
 & & & & &\\
\hline
    \end{tabular} 
    }
\end{center}
 Hence, \[A_{1}=K_{1}+K_{2}+K_{3}\]
 where
 \begin{itemize}
\item[a)]$K_{1}= \displaystyle \sum_{\substack{X\in S(0) \\ t=4 }}\Theta_{\theta} \left( \begin{bmatrix} m_1{} & X\\
 0 & m_{2}\end{bmatrix} \right)\overline{\psi_A({m_{1}}^{-1}X})=-q^{5}(1-q)(1-q^{2})(1-q^{3}).$\\
 \item[b)]$K_{2}= \displaystyle \sum_{\substack{X\in S(0) \\ t=3 }}\Theta_{\theta} \begin{bmatrix} m_{1} & X\\
 0 & m_{2}\end{bmatrix}\overline{\psi_A({m_{1}}^{-1}X})=-q^{5}(2q+1)(q-1)^{2}(q^{2}-1).$\\
 \item[c)] $ K_{3}= \displaystyle \sum_{\substack{X\in S(0) \\ t=2 }}\Theta_{\theta} \begin{bmatrix} m_{1} & X\\
 0 & m_{2}\end{bmatrix}\overline{\psi_A({m_{1}}^{-1}X})=q^{6}(q-1)^{3}.$
\end{itemize}
\vspace{-0.4 cm}
It follows that 
\begin{equation} \label{A_1(Type(1,1)} A_{1}=\sum_{X \in S(0)}\Theta_{\theta} \begin{bmatrix} m_{1} & X\\
 0 & m_{2}\end{bmatrix}\overline{\psi_A({m_{1}}^{-1}X})=q^{8}(q-1)^{3}.\end{equation}
Proceeding in a similar way, we find a partition of $S(\beta)$ to compute $A_{2}$. We record the details in the table below. 

\begin{center}
\captionof{table}{$A_2$}
\label{A_2}
\scriptsize
\scriptsize
\resizebox{1\textwidth}{!}{
    \begin{tabular}{| c | c | c | c | c | c |}
    \hline
  &  & &    &  &  \\
  & Partition of $S(\beta)$ & $\dim(XW')$ & $\dim(XW' \cap \Image(m_1-1))$ & $t=\dim(\Ker(h-1))$ & Cardinality\\
  & & & &  & \\
   \hline 
   
 & & & & &\\
$1) (a)$ & $\left \{\begin{bmatrix} 0 & 0 & e  \\ 0 & 0 & f\\
\beta & 0 & l \end{bmatrix} \right \}$ & $1$ & $0$ & $3$ & $q^3$\\
 & & & & & \\
\hline
& & & & & \\
$1)(b)$ & $\left \{ \begin{bmatrix} 0 & 0 & e\\
         0 & 0 & f\\
         \beta & k & l         \end{bmatrix} \,\middle \vert\, k \in F^{\times} \right \} $  & $1$ & $0$ & $3$ & $(q-1)q^3$\\
 & & & & &\\
\hline
& & & & & \\
$2)(a)$ & $\left \{ \begin{bmatrix} a & b & e\\
         c & d & f\\
         \beta & 0 & l         \end{bmatrix} \,\middle \vert\, d \in F^{\times}, ad-bc \neq 0 \right \} $  & $2$ & $0$ & $2$ & $(q-1)^2q^5$\\
 & & & & &\\
\hline
& & & & & \\
$2)(b)$ & $\left \{ \begin{bmatrix} a & b & e\\
         c & 0 & f\\
         \beta & 0 & l         \end{bmatrix} \,\middle \vert\,  bc \neq 0 \right \} $  & $2$ & $1$ & $3$ & $(q-1)^2q^4$\\
 & & & & &\\
\hline
& & & & & \\
$2)(c)$ & $\left \{ \begin{bmatrix} a & b & e\\
         0 & d & f\\
         \beta & k & l         \end{bmatrix} \,\middle \vert\, k \in F^{\times}, ad \neq 0 \right \} $  & $2$ & $0$ & $2$ & $(q-1)^2(q^2-q)q^3$\\
 & & & & &\\
\hline
& & & & & \\
$2)(d)$ & $\left \{ \begin{bmatrix} a & b & e\\
         c & 0 & f\\
         \beta & k & l         \end{bmatrix} \,\middle \vert\, k \in F^{\times}, bc \neq 0 \right \} $  & $2$ & $0$ & $2$ & $(q-1)^2(q^2-q)q^3$\\
 & & & & &\\
\hline
& & & & &\\
$2)(d)$ & $\left \{ \begin{bmatrix} a & b & e\\
         c & d & f\\
         \beta & k & l         \end{bmatrix} \,\middle \vert\, c,d,k \in F^{\times}, ad-bc \neq 0,d=\beta^{-1}ck \right \} $  & $2$ & $1$ & $3$ & $(q-1)^2(q^2-q)q^3$\\
 & & & & &\\
\hline
& & & & &\\
$2)(e)$ & $\left \{ \begin{bmatrix} a & b & e\\
         c & d & f\\
         \beta & k & l         \end{bmatrix} \,\middle \vert\, c,d,k \in F^{\times}, ad-bc \neq 0,d\neq\beta^{-1}ck \right \} $  & $2$ & $0$ & $2$ & $(q-1)^3(q-2)q^4$\\
 & & & & &\\
\hline
& & & & & \\
$3)(a)$ & $\left \{ \begin{bmatrix} a & \gamma a & e\\
         c & \gamma c & f\\
         \beta & 0 & l  \end{bmatrix} \,\middle \vert\, c,\gamma \in F^{\times} \right \} $  & $2$ & $0$ & $2$ & $(q-1)^2q^4$\\
 & & & & &\\
\hline
& & & & & \\
$3)(b)$ & $\left \{ \begin{bmatrix} a & \gamma a & e\\
         0 & 0 & f\\
         \beta & 0 & l         \end{bmatrix} \,\middle \vert\, a,\gamma \in F^{\times}\right \} $  & $2$ & $1$ & $3$ & $(q-1)^2q^3$\\
 & & & & &\\
\hline
& & & & & \\
$3)(c)$ & $\left \{ \begin{bmatrix} a & 0 & e\\
         c & 0 & f\\
         \beta & 0 & l         \end{bmatrix} \,\middle \vert\, (a,c) \neq 0 \right \} $  & $1$ & $0$ & $3$ & $(q^2-1)q^3$\\
 & & & & &\\
\hline
& & & & & \\
$3)(d)$ & $\left \{ \begin{bmatrix} a & \gamma a & e\\
         c & \gamma c & f\\
         \beta & k & l         \end{bmatrix} \,\middle \vert\, \gamma \in F^{\times}, k=\gamma \beta \right \} $  & $1$ & $0$ & $3$ & $(q^2-1)(q-1)q^3$\\
 & & & & &\\
\hline
& & & & & \\
$3)(e)$ & $\left \{ \begin{bmatrix} a & 0 & e\\
         c & 0 & f\\
         \beta & k & l         \end{bmatrix} \,\middle \vert\, k,c \in F^{\times} \right \} $  & $2$ & $0$ & $2$ & $(q-1)^2q^4$\\
 & & & & &\\
\hline
& & & & & \\
$3)(f)$ & $\left \{ \begin{bmatrix} a & 0 & e\\
         0 & 0 & f\\
         \beta & k & l         \end{bmatrix} \,\middle \vert\, k \in F^{\times} \right \} $  & $2$ & $1$ & $3$ & $(q-1)^2q^3$\\
 & & & & &\\
\hline
& & & & & \\
$3)(g)$ & $\left \{ \begin{bmatrix} a & \gamma a & e\\
         c & \gamma c & f\\
         \beta & k & l         \end{bmatrix} \,\middle \vert\, c,k,\gamma \in F^{\times}, k \neq \gamma \beta \right \} $  & $2$ & $0$ & $2$ & $(q-1)^2(q-2)q^4$\\
 & & & & &\\
\hline
& & & & & \\
$3)(h)$ & $\left \{ \begin{bmatrix} a & \gamma a & e\\
         0 & 0 & f\\
         \beta & k & l         \end{bmatrix} \,\middle \vert\, a,k,\gamma \in F^{\times}, k \neq \gamma \beta \right \} $  & $2$ & $1$ & $3$ & $(q-1)^2(q-2)q^3$\\
 & & & & &\\
\hline
& & & & & \\
$4)(a)$ & $\left \{ \begin{bmatrix} 0 & b & e\\
         0 & d & f\\
         \beta & k & l         \end{bmatrix} \,\middle \vert\, d \in F^{\times} \right \} $  & $2$ & $0$ & $2$ & $(q-1)q^5$\\
 & & & & &\\
\hline
& & & & & \\
$4)(b)$ & $\left \{ \begin{bmatrix} 0 & b & e\\
         0 & 0 & f\\
         \beta & k & l         \end{bmatrix} \,\middle \vert\, b \in F^{\times} \right \} $  & $2$ & $1$ & $3$ & $(q-1)q^4$\\
 & & & & &\\
\hline
    \end{tabular} 
    }
\end{center}
Hence, 
\[A_{2}=K_{4}+K_{5}\]
\begin{itemize}
\item[a)] $K_{4} = \displaystyle \sum_{\beta \in F^{\times}} \sum_{\substack{X\in S(\beta) \\ t=3 }}\Theta_{\theta} \begin{bmatrix} m_{1} & X\\
 0 & m_{2}\end{bmatrix}\overline{\psi_A({m_{1}}^{-1}X})=q^{7}(q-1)(q^{2}-1)$.\\
\item[b)] $K_{5}= \displaystyle \sum_{\beta \in F^{\times}} \sum_{\substack{X\in S(\beta) \\ t=2 }}\Theta_{\theta} \begin{bmatrix} m_{1} & X\\
 0 & m_{2}\end{bmatrix}\overline{\psi_A({m_{1}}^{-1}X})=-q^{7}(q-1)^{2}.$
\end{itemize}
It follows that 
\begin{equation} \label{A_2(Type(1,1)} A_{2}= \sum_{\beta \in F^{\times}}\sum_{X \in S(\beta)}\Theta_{\theta} \begin{bmatrix} m_{1} & X\\
 0 & m_{2}\end{bmatrix}\overline{\psi_A({m_{1}}^{-1}X})=q^{8}(q-1)^{2}.\end{equation}
 From \eqref{A_1(Type(1,1)} and \eqref{A_2(Type(1,1)}, it follows that 
 \[ \Theta_{N,\psi_A}(m)= (1-q)^2.\]
\end{proof}

\begin{theorem} \label{Type(1,2)}
Let $m \in T(1,2)$. Then, we have \[\Theta_{N,\psi_{A}}(m)=(1-q).\]
\end{theorem}

\begin{proof}
We have \[ \Theta_{N,\psi_{A}}(m)=\frac{1}{|N|} \sum_{ X \in M(3,F)}\Theta_{\theta} \begin{bmatrix} m_{1} & X\\
 0 & m_{2}\end{bmatrix}\overline{\psi_{A}({m_{1}}^{-1}X}).\]
Note that $\dim\Ker(m_1-1)=2$, $\Image(m_1-1)=\Span \left \{e_1 \right \}$ and $\Ker(m_2-1)=\Span \left \{ e_1 \right \}$.  
To calculate the character value, we write 
 \[ \Theta_{N,\psi_{A}}(m)=\frac{1}{q^{9}} (B_{1}+B_{2})\]
 where
 \[ B_{1}= \sum_{X \in S(0)}\Theta_{\theta} \begin{bmatrix} m_{1} & X\\
 0 & m_{2}\end{bmatrix}\overline{\psi_{A}({m_{1}}^{-1}X})\]
  and \[ B_{2}=\sum_{\beta \in F^{\times}}\sum_{X \in S(\beta)}\Theta_{\theta} \begin{bmatrix} m_{1} & X\\
 0 & m_{2}\end{bmatrix}\overline{\psi_{A}({m_{1}}^{-1}X}).\]
For simplicity, we let $t=\dim(\Ker(h-1))$. To compute $B_{1}$, we find a partition of $S(0)$ according to the value of $t$ and compute the respective cardinalities. We record the details in the table below. 

\begin{center}
\captionof{table}{$B_1$}
\label{B_1}
\scriptsize
\scriptsize
\resizebox{1\textwidth}{!}{
    \begin{tabular}{| c | c | c | c | c | c |}
    \hline
  &  & &    &  &  \\
  & Partition of $S(0)$ & $\dim(XW')$ & $\dim(XW' \cap \Image(m_1-1))$ & $t=\dim(\Ker(h-1))$ & Cardinality\\
  & & & &  & \\
   \hline 
   
 & & & & &\\
$1)$ & $\left \{\begin{bmatrix} 0 & b & e  \\ 0 & d & f\\
0 & k & l \end{bmatrix} \right \}$ & $0$ & $0$ & $3$ & $q^6$\\
 & & & & & \\
\hline
& & & & & \\
$2) (a)$ & $\left \{ \begin{bmatrix} a & b & e\\
         c & d & f\\
         0 & k & l         \end{bmatrix} \,\middle \vert\, c \in F^{\times} \right \} $  & $1$ & $0$ & $2$ & $(q-1)q^7$\\
 & & & & &\\
\hline
& & & & & \\
$2)(b)$ & $\left \{ \begin{bmatrix} a & b & e\\
         0 & d & f\\
         0 & k & l         \end{bmatrix} \,\middle \vert\, a \in F^{\times} \right \} $  & $1$ & $1$ & $3$ & $(q-1)q^6$\\
 & & & & &\\
\hline

    \end{tabular} 
    }
\end{center}
Hence,
\[B_1= K_{1}+K_{2}\]
where
\begin{itemize}
\item[a)]$K_{1}= \displaystyle \sum_{\substack{X\in S(0) \\ t=3 }}\Theta_{\theta} \begin{bmatrix} m_{1} & X\\
 0 & m_{2}\end{bmatrix}\overline{\psi_{A}({m_{1}}^{-1}X})=-q^{7}(1-q)(1-q^{2}).$\\
\item[b)]$K_{2}= \displaystyle \sum_{\substack{X\in S(0) \\ t=2 }}\Theta_{\theta} \begin{bmatrix} m_{1} & X\\
 0 & m_{2}\end{bmatrix}\overline{\psi_{A}({m_{1}}^{-1}X})=q^{7}(1-q)^{2}$.\\
\end{itemize}
It follows that \begin{equation} 
\label{B_1(Type(1,2)} B_{1}=\sum_{X \in S(0)}\Theta_{\theta} \begin{bmatrix} m_{1} & X\\
 0 & m_{2}\end{bmatrix} \overline{\psi_{A}({m_{1}}^{-1}X})=-q^{8}(q-1)^{2}. \end{equation}
Proceeding in a similar way, we find a partition of $S(\beta)$ to compute $B_{2}$. We record the details in the table below.  
\begin{center}
\captionof{table}{$B_{2}$}
\label{B_2}
\scriptsize
\scriptsize
\resizebox{1\textwidth}{!}{
    \begin{tabular}{| c | c | c | c | c | c |}
    \hline
  &  & &    &  &  \\
  & Partition of $S(\beta)$ & $\dim(XW')$ & $\dim(XW' \cap \Image(m_1-1))$ & $t=\dim(\Ker(h-1))$ & Cardinality\\
  & & & &  & \\
   \hline 
   
 & & & & &\\
$1)$ & $\left \{\begin{bmatrix} a & b & e  \\ c & d & f\\
\beta & k & l \end{bmatrix} \right \}$ & $1$ & $0$ & $2$ & $q^8$\\
 & & & & & \\
\hline
    \end{tabular} 
    }
\end{center}
Hence,
\begin{equation}\label{B_2(Type(1,2)} B_{2}=\displaystyle \sum_{\beta \in F^{\times}} \sum_{\substack{X\in S(\beta) \\ t=2 }}\Theta_{\theta} \begin{bmatrix} m_{1} & X\\0 & m_{2}\end{bmatrix}\overline{\psi_{A}({m_{1}}^{-1}X})=q^{8}(1-q).\end{equation}
From \eqref{B_1(Type(1,2)} and \eqref{B_2(Type(1,2)}, it follows that 
 \[ \Theta_{N,\psi_A}(m)= (1-q).\]
 \end{proof}

\begin{theorem} \label{Type(4,1)}
Let $m \in T(4,1)$. Then, we have \[\Theta_{N,\psi_A}(m)=(1-q)^2(1-q^2).\]
\end{theorem}
 
\begin{proof} We have 
\[\Theta_{N,\psi_A}(m)=\frac{1}{|N|} \sum_{ X \in M(3,F)}\Theta_{\theta} \begin{bmatrix} m_{1} & X\\
 0 & m_{2}\end{bmatrix}\overline{\psi_{A}({m_{1}}^{-1}X}).\]
Note that $\dim\Ker(m_1-1)=3$, $\Image(m_1-1)= \{0\}$ and $\Ker(m_2-1)=\Span \left \{ e_1,e_2 \right \}$.
To calculate the character value, we write 
 \[ \Theta_{N,\psi_{A}}(m)=\frac{1}{q^{9}} (C_{1}+C_{2})\]
where we have 
 \[ C_{1}= \sum_{X \in S(0)}\Theta_{\theta} \begin{bmatrix} m_{1} & X\\
 0 & m_{2}\end{bmatrix}\overline{\psi_{A}({m_{1}}^{-1}X})\]
  and \[ C_{2}=\sum_{\beta \in F^{\times}}\sum_{X \in S(\beta)}\Theta_{\theta} \begin{bmatrix} m_{1} & X\\
 0 & m_{2}\end{bmatrix}\overline{\psi_{A}({m_{1}}^{-1}X}).\]
 For simplicity, we let $t=\dim(\Ker(h-1))$. To compute $C_{1}$, we find a partition of $S(0)$ according to the value of $t$ and compute the respective cardinalities. We record the details in the following table. 
 \begin{center}
\captionof{table}{$C_1$}
\label{C_1}
\scriptsize
\scriptsize
\resizebox{1\textwidth}{!}{
    \begin{tabular}{| c | c | c | c | c | c |}
    \hline
  &  & &    &  &  \\
  & Partition of $S(0)$ & $\dim(XW')$ & $\dim(XW' \cap \Image(m_1-1))$ & $t=\dim(\Ker(h-1))$ & Cardinality\\
  & & & &  & \\
   \hline 
   
 & & & & &\\
$1) (a)$ & $\left \{\begin{bmatrix} 0 & 0 & e  \\ 0 & 0 & f\\
0 & 0 & l \end{bmatrix} \right \}$ & $0$ & $0$ & $5$ & $q^3$\\
 & & & & & \\
\hline
& & & & & \\
$1)(b)$ & $\left \{ \begin{bmatrix} 0 & 0 & e\\
         0 & 0 & f\\
         0 & k & l         \end{bmatrix} \,\middle \vert\, k \in F^{\times} \right \} $  & $1$ & $0$ & $4$ & $(q-1)q^3$\\
 & & & & &\\
\hline
& & & & & \\
$2)(a)$ & $\left \{ \begin{bmatrix} a & b & e\\
         c & d & f\\
         0 & 0 & l         \end{bmatrix} \,\middle \vert\, ad-bc \neq 0 \right \} $  & $2$ & $0$ & $3$ & $(q^2-1)(q^2-q)q^3$\\
 & & & & &\\
\hline
& & & & & \\
$2)(b)$ & $\left \{ \begin{bmatrix} a & b & e\\
         c & d & f\\
         0 & k & l         \end{bmatrix} \,\middle \vert\, k \in F^{\times}, ad-bc \neq 0 \right \} $  & $2$ & $0$ & $3$ & $(q^2-1)(q-1)^2q^4$\\
 & & & & &\\
\hline
& & & & & \\
$3)(a)$ & $\left \{ \begin{bmatrix} a & \gamma a & e\\
         c & \gamma c & f\\
         0 & 0 & l         \end{bmatrix} \,\middle \vert\, \gamma \in F \right \} $  & $1$ & $0$ & $4$ & $(q^2-1)q^4$\\
 & & & & &\\
\hline
& & & & & \\
$3)(b)$ & $\left \{ \begin{bmatrix} a & \gamma a & e\\
         c & \gamma c & f\\
         0 & k & l         \end{bmatrix} \,\middle \vert\, k \in F^{\times}, \gamma \in F \right \} $  & $2$ & $0$ & $3$ & $(q^2-1)(q-1)q^4$\\
 & & & & &\\
\hline
& & & & & \\
$4)(a)$ & $\left \{ \begin{bmatrix} 0 & b & e\\
         0 & d & f\\
         0 & 0 & l         \end{bmatrix} \right \} $  & $1$ & $0$ & $4$ & $(q^2-1)q^3$\\
 & & & & &\\
\hline
& & & & &\\
$4)(b)$ & $\left \{ \begin{bmatrix} 0 & b & e\\
         0 & d & f\\
         0 & k & l         \end{bmatrix} \,\middle \vert\, k \in F^{\times} \right \} $  & $1$ & $0$ & $4$ & $(q^2-1)(q-1)q^3$\\
 & & & & &\\
\hline
    \end{tabular} 
    }
\end{center}
Hence, 
 \[C_1= K_1+K_2+K_3\]
 where
 \begin{itemize}
\item[a)]$K_1= \displaystyle \sum_{\substack{X\in S(0) \\ t= 5}}\Theta_{\theta} \begin{bmatrix} m_1 & X\\
 0 & m_2\end{bmatrix}\overline{\psi_A({m_1}^{-1}X})=-q^{3}(1-q)(1-q^2)(1-q^3)(1-q^4)$.\\
\item[b)]$K_2 = \displaystyle \sum_{\substack{X\in S(0) \\ t= 4}}\Theta_{\theta} \begin{bmatrix} m_1 & X\\
 0 & m_2\end{bmatrix}\overline{\psi_A({m_1}^{-1}X})=q^3(1-q)^2(1-q^2)(1-q^3)(2q^2+2q+1)$.\\
\item[c)] $K_3= \displaystyle \sum_{\substack{X\in S(0) \\ t= 3 }}\Theta_{\theta} \begin{bmatrix} m_1 & X\\
 0 & m_2\end{bmatrix}\overline{\psi_A({m_1}^{-1}X})=-q^{4}(q^2-1)^2(1-q)(1-q^2).$\\
\end{itemize}
It follows that \begin{equation} \label{C_1(Type(4,1)} C_1=\sum_{X \in S(0)}\Theta_{\theta} \begin{bmatrix} m_1 & X\\
 0 & m_2\end{bmatrix}\overline{\psi_A({m_1}^{-1}X})=-q^{8}(1-q)^3(1-q^2).\end{equation}
Proceeding in a similar way, we find a partition of $S(\beta)$ to compute $C_{2}$. We record the details in the table below. 
\begin{center}
\captionof{table}{$C_2$}
\label{C_2}
\scriptsize
\scriptsize
\resizebox{1\textwidth}{!}{
    \begin{tabular}{| c | c | c | c | c | c |}
    \hline
  &  & &    &  &  \\
  & Partition of $S(\beta)$ & $\dim(XW')$ & $\dim(XW' \cap \Image(m_1-1))$ & $t=\dim(\Ker(h-1))$ & Cardinality\\
  & & & &  & \\
   \hline 
   
 & & & & &\\
$1) (a)$ & $\left \{\begin{bmatrix} 0 & 0 & e  \\ 0 & 0 & f\\
\beta & 0 & l \end{bmatrix} \right \}$ & $1$ & $0$ & $4$ & $q^3$\\
 & & & & & \\
\hline
& & & & & \\
$1)(b)$ & $\left \{ \begin{bmatrix} 0 & 0 & e\\  
         0 & 0 & f\\
         \beta & k & l         \end{bmatrix} \,\middle \vert\, k \in F^{\times} \right \} $  & $1$ & $0$ & $4$ & $(q-1)q^3$\\
 & & & & &\\
\hline
& & & & & \\
$2)(a)$ & $\left \{ \begin{bmatrix} a & b & e\\
         c & d & f\\
         \beta & 0 & l         \end{bmatrix} \,\middle \vert\, ad-bc \neq 0 \right \} $  & $2$ & $0$ & $3$ & $(q^2-1)(q^2-q)q^3$\\
 & & & & &\\
\hline
& & & & & \\
$2)(b)$ & $\left \{ \begin{bmatrix} a & b & e\\
         c & d & f\\
         \beta & k & l         \end{bmatrix} \,\middle \vert\, k \in F^{\times},ad-bc \neq 0 \right \} $  & $2$ & $0$ & $3$ & $(q^2-1)(q-1)^2q^4$\\
 & & & & &\\
\hline
& & & & & \\
$3)(a)$ & $\left \{ \begin{bmatrix} a & \gamma a & e\\
         c & \gamma c & f\\
         \beta & 0 & l  \end{bmatrix} \,\middle \vert\, \gamma \in F^{\times} \right \} $  & $2$ & $0$ & $3$ & $(q^2-1)(q-1)q^3$\\
 & & & & &\\
\hline
& & & & & \\
$3)(b)$ & $\left \{ \begin{bmatrix} a & 0 & e\\
         c & 0 & f\\
         \beta & 0 & l         \end{bmatrix} \,\middle \vert\, (a,c) \neq 0 \right \} $  & $1$ & $0$ & $4$ & $(q^2-1)q^3$\\
 & & & & &\\
\hline
& & & & & \\
$3)(c)$ & $\left \{ \begin{bmatrix} a & 0 & e\\
         c & 0 & f\\
         \beta & k & l         \end{bmatrix} \,\middle \vert\, (a,c) \neq 0, k \in F^{\times} \right \} $  & $2$ & $0$ & $3$ & $(q^2-1)(q-1)q^3$\\
 & & & & &\\
\hline
& & & & & \\
$3)(d)$ & $\left \{ \begin{bmatrix} a & \gamma a & e\\
         c & \gamma c & f\\
         \beta & k & l         \end{bmatrix} \,\middle \vert\, \gamma \in F^{\times}, k=\gamma \beta \right \} $  & $1$ & $0$ & $4$ & $(q^2-1)(q-1)q^3$\\
 & & & & &\\
\hline
& & & & & \\
$3)(e)$ & $\left \{ \begin{bmatrix} a & \gamma a & e\\
         c & \gamma c & f\\
         \beta & k & l         \end{bmatrix} \,\middle \vert\, k, \gamma \in F^{\times}, k \neq \gamma \beta \right \} $  & $2$ & $0$ & $3$ & $(q^2-1)(q-1)(q-2)q^3$\\
 & & & & &\\
\hline
& & & & &\\
$4)(a)$ & $\left \{ \begin{bmatrix} 0 & b & e\\
         0 & d & f\\
         \beta & 0 & l         \end{bmatrix} \,\middle \vert\, (b,d) \neq 0 \right \} $  & $2$ & $0$ & $3$ & $(q^2-1)q^3$\\
 & & & & &\\
\hline
& & & & & \\
$4)(b)$ & $\left \{ \begin{bmatrix} 0 & b & e\\
         0 & d & f\\
         \beta & k & l         \end{bmatrix} \,\middle \vert\, (b,d) \neq 0, k \in F^{\times} \right \} $  & $2$ & $0$ & $3$ & $(q^2-1)(q-1)q^3$\\
 & & & & &\\
\hline
    \end{tabular}
    }
    \end{center}
We have 
\[ C_2=K_{4}+K_{5}\]
where 
\begin{itemize}
\item[a)] $ K_{4}= \displaystyle \sum_{\beta \in F^{\times}} \sum_{\substack{X\in S(\beta) \\ t=4 }}\Theta_{\theta} \begin{bmatrix} m_1 & X\\
 0 & m_2\end{bmatrix}\overline{\psi_A({m_1}^{-1}X})=q^6(1-q)(1-q^2)(1-q^3).$\\
 \item[b)] $ K_5= \displaystyle \sum_{\beta \in F^{\times}} \sum_{\substack{X\in S(\beta) \\ t=3 }}\Theta_{\theta} \begin{bmatrix} m_1 & X\\
 0 & m_2\end{bmatrix}\overline{\psi_A({m_1}^{-1}X})=-q^6(1-q)(1-q^2)^2.$
\end{itemize}
It follows that 
\begin{equation} \label{C_2(Type(4,1)} C_2= \sum_{\beta \in F^{\times}}\sum_{X \in S(\beta)}\Theta_{\theta} \begin{bmatrix} m_1 & X\\
 0 & m_2\end{bmatrix}\overline{\psi_A({m_1}^{-1}X})=q^8(1-q)^2(1-q^2).\end{equation}
 From \eqref{C_1(Type(4,1)} and \eqref{C_2(Type(4,1)}, we have 
 \[ \Theta_{N,\psi_A}(m)= (1-q^2)(1-q)^2.\]
 \end{proof}
 
\begin{remark}
Let $m\in T(1,4)$. Since \[\Theta_{N,\psi_{A}}(m)=\Theta_{N,\psi_{A}}(m')\] for some $m' \in T(4,1)$, it is enough to compute $\Theta_{N,\psi_{A}}(m')$ for $m' \in T(4,1)$ to obtain the character value $\Theta_{N,\psi_A}(m)$.
\end{remark}

\begin{theorem} \label{Type(2,2)}
Let $m \in T(2,2)$.  Then, we have \[\Theta_{N,\psi_{A}}(m)=1.\]
\end{theorem}

\begin{proof}
 We have \[ \Theta_{N,\psi_A}(m)=\frac{1}{|N|} \sum_{ X \in M(3,F)}\Theta_{\theta} \begin{bmatrix} m_1 & X\\
 0 & m_2\end{bmatrix}\overline{\psi_A({m_1}^{-1}X}).\]
 Note that $\dim\Ker(m_1-1)=1$, $\Image(m_1-1)=\Span \left \{ e_1,e_2 \right \}$ and $\Ker(m_2-1)=\Span \left \{ e_1 \right \}$. To calculate the character value, we write 
 \[ \Theta_{N,\psi_A}(m)=\frac{1}{q^9} (D_1+D_2)\]
 where 
 \[ D_1= \sum_{X \in S(0)}\Theta_{\theta} \begin{bmatrix} m_1 & X\\
 0 & m_2\end{bmatrix}\overline{\psi_A({m_1}^{-1}X})\]
  and \[ D_2=\sum_{\beta \in F^{\times}}\sum_{X \in S(\beta)}\Theta_{\theta} \begin{bmatrix} m_1 & X\\
 0 & m_2\end{bmatrix}\overline{\psi_A({m_1}^{-1}X}).\]
 Let $t=\dim(\Ker(h-1))$. To compute $D_{1}$ we find a partition of $S(0)$ according to the value of $t$ and compute the respective cardinalities. We record the details in the table below.  
 \begin{center}
\captionof{table}{$D_1$}
\label{D_1}
\scriptsize
\scriptsize
\resizebox{1\textwidth}{!}{
    \begin{tabular}{| c | c | c | c | c | c |}
    \hline
  &  & &    &  &  \\
  & Partition of $S(0)$ & $\dim(XW')$ & $\dim(XW' \cap \Image(m_1-1))$ & $t=\dim(\Ker(h-1))$ & Cardinality\\
  & & & &  & \\
   \hline 
   
 & & & & & \\
$1) (a)$ & $\left \{\begin{bmatrix} a & b & e  \\ c & d & f\\
0 & k & l \end{bmatrix} \mid (a,c) \neq 0 \right \}$ & $1$ & $1$ & $2$ & $(q^2-1)q^6$\\
 & & & & & \\
\hline
& & & & & \\
$1)(b)$ & $\left \{ \begin{bmatrix} 0 & b & e\\
         0 & d & f\\
         0 & k & l         \end{bmatrix} \right \} $  & $0$ & $0$ & $2$ & $q^6$\\
 & & & & &\\
\hline
    \end{tabular} 
    }
\end{center}
Hence,
 \begin{equation} \label{D_1(Type(2,2)} D_1=\displaystyle \sum_{\substack{X\in S(0) \\ t=2 }}\Theta_{\theta} \begin{bmatrix} m_1 & X\\0 & m_2\end{bmatrix}\overline{\psi_A({m_1}^{-1}X})=-q^8(1-q).
 \end{equation} 
Proceeding in a similar way, we find a partition of $S(\beta)$ to compute $D_{2}$. We record the details in the following table. 
\begin{center}
\captionof{table}{$D_2$}
\label{D_2}
\scriptsize
\scriptsize
\resizebox{1\textwidth}{!}{
    \begin{tabular}{| c | c | c | c | c | c |}
    \hline
  &  & &    &  &  \\
  & Partition of $S(\beta)$ & $\dim(XW')$ & $\dim(XW' \cap \Image(m_1-1))$ & $t=\dim(\Ker(h-1))$ & Cardinality\\
  & & & &  & \\
   \hline 
   
 & & & & & \\
$1)$ & $\left \{\begin{bmatrix} a & b & e  \\ c & d & f\\
\beta & k & l \end{bmatrix} \right \}$ & $1$ & $0$ & $1$ & $q^8$\\
 & & & & & \\
\hline
    \end{tabular} 
    }
\end{center}
Thus, we have
\begin{equation}\label{D_2(Type(2,2)} D_2=\displaystyle \sum_{\beta \in F^{\times}} \sum_{\substack{X\in S(\beta) \\ t=1 }}\Theta_{\theta} \begin{bmatrix} m_1 & X\\0 & m_2\end{bmatrix}\overline{\psi_A({m_1}^{-1}X})=q^8.\end{equation}

 From \eqref{D_1(Type(2,2)} and \eqref{D_2(Type(2,2)}, it follows that 
 \[ \Theta_{N,\psi_A}(m)= 1.\]
 \end{proof}

\begin{theorem} \label{Type(4,2)}
 Let $m \in T(4,2)$. Then, we have \[\Theta_{N,\psi_A}(m)=(1-q)(1-q^2).\]
 \end{theorem}
\begin{proof}
 
 We have \[ \Theta_{N,\psi_A}(m)=\frac{1}{|N|} \sum_{ X \in M(3,F)}\Theta_{\theta} \begin{bmatrix} m_1 & X\\
 0 & m_2\end{bmatrix}\overline{\psi_A({m_1}^{-1}X}).\]
Note that $\dim\Ker(m_1-1)=3$, $\Image(m_1-1)=\{0\}$ and $\Ker(m_2-1)=\Span \{e_1\}$. To calculate the character value, we write 
 \[ \Theta_{N,\psi_A}(m)=\frac{1}{q^9} (H_1+H_2)\]
 where
 \[ H_1= \sum_{X \in S(0)}\Theta_{\theta} \begin{bmatrix} m_1 & X\\
 0 & m_2\end{bmatrix}\overline{\psi_A({m_1}^{-1}X})\]
  and \[ H_2=\sum_{\beta \in F^{\times}}\sum_{X \in S(\beta)}\Theta_{\theta} \begin{bmatrix} m_1 & X\\
 0 & m_2\end{bmatrix}\overline{\psi_A({m_1}^{-1}X}).\]
Let $t=\dim(\Ker(h-1))$. To compute $H_{1}$, we find a partition of $S(0)$ according to the value of $t$ and compute the respective cardinalities. We record the details in the table below. 
\begin{center}
\captionof{table}{$H_1$}
\label{H_1}
\scriptsize
\scriptsize
\resizebox{1\textwidth}{!}{
    \begin{tabular}{| c | c | c | c | c | c |}
    \hline
  &  & &    &  &  \\
  & Partition of $S(0)$ & $\dim(XW')$ & $\dim(XW' \cap \Image(m_1-1))$ & $t=\dim(\Ker(h-1))$ & Cardinality\\
  & & & &  & \\
   \hline 
   
 & & & & & \\
$1) (a)$ & $\left \{\begin{bmatrix} a & b & e  \\ c & d & f\\
0 & k & l \end{bmatrix} \,\middle \vert\, (a,c) \neq 0 \right \}$ & $1$ & $0$ & $3$ & $(q^2-1)q^6$\\
 & & & & & \\
\hline
& & & & & \\
$1)(b)$ & $\left \{ \begin{bmatrix} 0 & b & e\\
         0 & d & f\\
         0 & k & l         \end{bmatrix} \right \} $  & $0$ & $0$ & $4$ & $q^6$\\
 & & & & &\\
\hline
    \end{tabular} 
    }
\end{center}
 Hence, 
 \[H_1= K_1+K_2\]
 where
 \begin{itemize}
\item[a)]$K_1= \displaystyle \sum_{\substack{X\in S(0) \\ t=4 }}\Theta_{\theta} \begin{bmatrix} m_1 & X\\
 0 & m_2\end{bmatrix}\overline{\psi_A({m_1}^{-1}X})=-q^6(1-q)(1-q^2)(1-q^3)$.\\
 \item[b)]$K_2= \displaystyle \sum_{\substack{X\in S(0) \\ t=3 }}\Theta_{\theta} \begin{bmatrix} m_1 & X\\
 0 & m_2\end{bmatrix}\overline{\psi_A({m_1}^{-1}X})=q^6(1-q)(1-q^2)^2$.\\
\end{itemize}
It follows that \begin{equation} \label{H_1(Type(4,2)} H_1=\sum_{X \in S(0)}\Theta_{\theta} \begin{pmatrix} m_1 & X\\
 0 & m_2\end{pmatrix}\overline{\psi_A({m_1}^{-1}X})=-q^8(1-q)^2(1-q^2).\end{equation}
Proceeding in a similar way, we find a partition of $S(\beta)$ to compute $H_{2}$. We record the details in the following table. 
\begin{center}
\captionof{table}{$H_2$}
\label{H_2}
\scriptsize
\scriptsize
\resizebox{1\textwidth}{!}{
    \begin{tabular}{| c | c | c | c | c | c |}
    \hline
  &  & &    &  &  \\
  & Partition of $S(\beta)$ & $\dim(XW')$ & $\dim(XW' \cap \Image(m_1-1))$ & $t=\dim(\Ker(h-1))$ & Cardinality\\
  & & & &  & \\
   \hline 
   
 & & & & & \\
$1)$ & $\left \{\begin{bmatrix} a & b & e  \\ c & d & f\\
\beta & k & l \end{bmatrix} \right \}$ & $1$ & $0$ & $3$ & $q^8$\\
 & & & & & \\
\hline
    \end{tabular} 
    }
\end{center}
Thus, we have 
\begin{equation}\label{H_2(Type(4,2)} H_2=\displaystyle \sum_{\beta \in F^{\times}} \sum_{\substack{X\in S(\beta) \\ t=3 }}\Theta \begin{bmatrix} m_1 & X\\0 & m_2\end{bmatrix}\overline{\psi_A({m_1}^{-1}X})=q^8(1-q)(1-q^2).\end{equation}
From \eqref{H_1(Type(4,2)} and \eqref{H_2(Type(4,2)}, it follows that 
 \[ \Theta_{N,\psi_A}= (1-q)(1-q^2).\]
 \end{proof}
 
\begin{remark}
 Let $m \in T(2,4)$. Since \[\Theta_{N,\psi_A}(m)=\Theta_{N,\psi_A}(m')\] for some $m' \in T(4,2)$, it is enough to compute $\Theta_{N,\psi_A}(m')$ for $m' \in T(4,2)$ to obtain $\Theta_{N,\psi_A}(m)$.
\end{remark}
 
\begin{theorem} \label{Type(4,4)}
Let $m \in T(4,4)$. Then, we have \[\Theta_{N,\psi_A}(m)=(1-q)^2(1-q^2)^2.\]
\end{theorem}
\begin{proof} Since $m\in T(4,4)$, we have $m=1$, and the result follows from Theorem~\ref{dimension calculation}. To be precise, we have
\[\Theta_{N,\psi_{A}}(m)= \dim_{\mathbb{C}}(\pi_{N,\psi_{A}})= (1-q)^2(1-q^2)^2.\] 
\end{proof}

\begin{theorem} \label{Type(x,5)}
Let $1 \leq i \leq 10$. Suppose that $m=\begin{bmatrix} m_1 & 0 \\
0 & m_2 \end{bmatrix} \in T(i,5)$ and $m'=\begin{bmatrix} m_1 & 0\\
0 & m_2' \end{bmatrix} \in T(i,1)$. Then, we have 
\[\Theta_{N,\psi_A}(m)=\Theta_{N,\psi_A}(m').\]
\end{theorem}
 
\begin{proof}
Let $h=\begin{bmatrix} m_1 & X \\
0 & m_2 \end{bmatrix}$ and $h'=\begin{bmatrix} m_1 & X \\
0 &m_2' \end{bmatrix}$ for $X \in \M(3,F)$. Let \[M_{m_{1},m_{2}}^{ d,\beta}=\{X \in S(\beta) \mid \dim(\Ker(h-1))=d\}\] for $\beta \in F$.
Clearly,
\[\Ker(m_2-1)=\Ker(m_2'-1).\]
Hence for any $X \in \M(3,F)$, 
\[X\Ker(m_2-1)=X\Ker(m_2'-1)\]
and \[X\Ker(m_2-1) \cap \Image(m_1-1)=X\Ker(m_2'-1) \cap \Image(m_1-1).\]
In particular, for any $\beta \in F$, we have that \begin{equation}\label{M_d^h}M_{m_{1},m_{2}}^{ d,\beta}=M_{m_{1},m_{2}'}^{ d,\beta}.
\end{equation}
We have, 
\[\Theta_{N,\psi_A}(m)=\frac{1}{q^9}(R_1+R_2)\]
where
\begin{align*} R_1&=\sum_{d=1}^{6}\sum_{ X \in M_{m_{1},m_{2}}^{ d,0}}\Theta_{\theta}(h)\overline{\psi_A({m_1}^{-1}X)}\\
&= \sum_{d=1}^{6}(-1)^{6-1}(1-q) \cdots (1-q^{d-1})(\# M_{m_{1},m_{2}}^{ d,0})\overline{\psi_0(0)} \end{align*}
and \begin{align*} R_2 &=\sum_{\beta \in F^{\times}}\sum_{d=1}^{6}\sum_{X \in M_{m_{1},m_{2}}^{ d,\beta}}\Theta_{\theta}(h)\overline{\psi_A({m_1}^{-1}X)}\\
 &=\sum_{\beta \in F^{\times}} \sum_{d=1}^{6}(-1)^{6-1}(1-q) \cdots (1-q^{d-1})(\# M_{m_{1},m_{2}}^{ d,\beta})\overline{\psi_0(\beta)} \end{align*}
Thus,  \[\Theta_{N,\psi_A}(m)=\frac{1}{q^9}\sum_{d=1}^{6}(-1)^{6-1}(1-q) \cdots (1-q^{d-1})(\# M_{m_{1},m_{2}}^{ d,0}+ \sum_{\beta \in F^{\times}}\# M_{m_{1},m_{2}}^{ d,\beta}\overline{\psi_0(\beta})).\]
Similarly, \[\Theta_{N,\psi_A}(m')=\frac{1}{q^9}\sum_{d=1}^{6}(-1)^{6-1}(1-q) \cdots (1-q^{d-1})(\# M_{m_{1},m_{2}'}^{ d,0}+ \sum_{\beta \in F^{\times}}\# M_{m_{1},m_{2}'}^{ d,\beta}\overline{\psi_0(\beta})).\]
Hence, it follows from equation ~\eqref{M_d^h} that
\[\Theta_{N,\psi_A}(m)=\Theta_{N,\psi_A}(m').\]
 \end{proof}

\begin{proposition}\label{Type(x,3)}
Let $1 \leq i \leq 10$ and $m\in T(i,3)$. Then, there exists some $m'\in T(i,5)$ such that \[\Theta_{N,\psi_A}(m)=\Theta_{N,\psi_A}(m').\] 
\end{proposition}

\begin{proof} Let $m=\begin{bmatrix} m_{1} & 0 \\ 0 & m_{2}\end{bmatrix}\in T(i,3)$, and  $w= \begin{bmatrix} 1 & 0 & 0\\
 0 & 0 & 1\\
 0 & 1 & 0 \end{bmatrix}$. Let $m_{2}'= wm_{2}w^{-1}$ and $m'=\begin{bmatrix} m_1 & 0\\
0 & m_2' \end{bmatrix}$. Clearly, we have $m= \begin{bmatrix}1 & 0\\
0 & w \end{bmatrix} m' {\begin{bmatrix} 1 & 0\\
0 & w \end{bmatrix}}^{-1}$ and $m^{'}\in T(i,5)$. Hence the result. 
\end{proof}
   
\begin{corollary}
Let $1 \leq i \leq 10$ and $m\in T(i,3)$.  Then, \[\Theta_{N,\psi_A}(m)=\Theta_{N,\psi_A}(m')\] for some $m' \in T(i,1)$.
\end{corollary}

\begin{proof}
Using Proposition \ref{Type(x,3)} and Theorem \ref{Type(x,5)}, the result follows.
\end{proof}

\begin{proposition} \label{(x,7)}
Let $1\leq i\leq 10$ and $m\in T(i,7)$. Then, there exists some $m'\in T(i,1)$ such that 
\[\Theta_{N,\psi_A}(m)=\Theta_{N,\psi_A}(m'). \]
\end{proposition}

\begin{proof} Let $m=\begin{bmatrix} m_{1} & 0 \\ 0 & m_{2}\end{bmatrix}\in T(i,7)$, and  $w= \begin{bmatrix} 1 & 0 & 0\\
 0 & 0 & 1\\
 0 & 1 & 0 \end{bmatrix}$. Let $m_{2}'= wm_{2}w^{-1}$ and $m'=\begin{bmatrix} m_1 & 0\\
0 & m_2' \end{bmatrix}$. Clearly, we have $m= \begin{bmatrix}1 & 0\\
0 & w \end{bmatrix} m' {\begin{bmatrix} 1 & 0\\
0 & w \end{bmatrix}}^{-1}$ and $m^{'}\in T(i,1)$. Hence the result. 
\end{proof}

\begin{proposition} \label{(x,8)}
Let $1\leq i\leq 10$ and $m\in T(i,8)$. Then, there exists some $m'\in T(i,2)$ such that 
\[\Theta_{N,\psi_A}(m)=\Theta_{N,\psi_A}(m'). \]
\end{proposition}

\begin{proof} Let $m=\begin{bmatrix} m_{1} & 0 \\ 0 & m_{2}\end{bmatrix}\in T(i,8)$, and  $w= \begin{bmatrix} 1 & 0 & 0\\
 0 & 0 & 1\\
 0 & 1 & 0 \end{bmatrix}$. Let $m_{2}'= wm_{2}w^{-1}$ and $m'=\begin{bmatrix} m_1 & 0\\
0 & m_2' \end{bmatrix}$. Clearly, we have $m= \begin{bmatrix}1 & 0\\
0 & w \end{bmatrix} m' {\begin{bmatrix} 1 & 0\\
0 & w \end{bmatrix}}^{-1}$ and $m^{'}\in T(i,2)$. Hence the result. 
\end{proof}

\begin{theorem} \label{(1,6)}
Let $m=\begin{bmatrix} m_1 & 0\\
0 & m_2 \end{bmatrix}\in $ $T(1,6)$ or $T(1,9)$. Then, we have 
\[\Theta_{N,\psi_A}(m)=(1-q)^2.\] 
\end{theorem}
 
\begin{proof} Note that $\dim\Ker(m_1-1)=2$, $\Image(m_1-1)=\Span\{e_1\}$, $\dim \Ker(m_2-1)=2$. From Remark \ref{X}, it follows that whenever \[h=\begin{bmatrix} m_1 & X\\
0 & m_2 \end{bmatrix},\] we have
\[XW'=\Span\{X_1,X_2\}.\]
Thus, proceeding in a similar fashion as in Theorem \ref{Type(1,1)}, we get that 
\[\Theta_{N,\psi_A}(m)=(1-q)^2.\]
\end{proof}

\begin{theorem} \label{(2,6)}
Let $m=\begin{bmatrix} m_1 & 0\\
0 & m_2 \end{bmatrix}\in $ $T(2,6)$ or $T(2,9)$. Then, we have 
\[\Theta_{N,\psi_A}(m)=(1-q).\]
\end{theorem}
 
\begin{proof}
Note that $\dim\Ker(m_1-1)=1$, $\Image(m_1-1)=\{e_1,e_2\}$, $\dim \Ker(m_2-1)=2$. From Remark \ref{X}, it follows that whenever \[h=\begin{bmatrix} m_1 & X\\
0 & m_2 \end{bmatrix},\] we have
\[XW'=\Span\{X_1,X_2\}.\]
Thus, the computations are similar to the case where $m \in T(2,1)$. The result follows from Theorem \ref{T(i,j)=T(j,i)} and Theorem \ref{Type(1,2)}.
\end{proof}

\begin{theorem} \label{(4,6)}
Let $m=\begin{bmatrix} m_1 & 0\\
0 & m_2 \end{bmatrix}\in $  $T(4,6)$ or $T(4,9)$. Then, we have 
\[\Theta_{N,\psi_A}(m)=(1-q)^2(1-q^2).\] 
\end{theorem}
 
\begin{proof} Note that $\dim\Ker(m_1-1)=3$, $\Image(m_1-1)=\{0\}$, $\dim \Ker(m_2-1)=2$. From Remark \ref{X}, it follows that whenever \[h=\begin{bmatrix} m_1 & X\\
0 & m_2 \end{bmatrix},\] we have
\[XW'=\Span\{X_1,X_2\}.\]
Thus, proceeding in a similar fashion as in Theorem \ref{Type(4,1)}, we get that 
\[\Theta_{N,\psi_A}(m)=(1-q)^2(1-q^2).\]
\end{proof}

\begin{theorem} \label{(6,6)}
Let $m=\begin{bmatrix} m_1 & 0\\
0 & m_2 \end{bmatrix} \in T(6,6)$ or $T(6,9)$. Then, we have 
\[\Theta_{N,\psi_A}(m)=(1-q)^2.\] 
\end{theorem}
 
\begin{proof} Note that $\dim\Ker(m_1-1)=2$, $\dim\Ker(m_2-1)=2$, $\Image(m_1-1)=\Span\{\eta e_1+e_2\}$ 
for $\eta \in F^{\times}$. From Remark \ref{X}, it follows that computing $\Theta_{N,\psi_A}(m)$ for $m \in T(6,6)$ or $m\in T(6,9)$ is the same as computing $\Theta_{N,\psi_A}(m')$ for $m' \in T(6,1)$. Using Theorem \ref{T(i,j)=T(j,i)} and Theorem \ref{(1,6)}, it follows that 
\[\Theta_{N,\psi_A}(m)=(1-q)^2.\]
\end{proof}

\begin{theorem}
Let $m=\begin{bmatrix} m_1 & 0\\
0 & m_2 \end{bmatrix} \in T(9,9)$.
Then, we have 
\[\Theta_{N,\psi_A}(m)=(1-q)^2.\] 
\end{theorem}
 
\begin{proof}
The proof is similar to Theorem \ref{(6,6)}.
\end{proof}

\begin{theorem} \label{Type(x,10)}
Let $1 \leq i \leq 10$. Suppose $m=\begin{bmatrix} m_1 & 0 \\
0 & m_2 \end{bmatrix} \in T(i,10)$ and $m'=\begin{bmatrix} m_1 & 0\\
0 & m_2' \end{bmatrix} \in T(i,2)$. Then, we have 
\[\Theta_{N,\psi_A}(m)=\Theta_{N,\psi_A}(m').\]
\end{theorem}

\begin{proof}
Let $h=\begin{bmatrix} m_1 & X \\
0 & m_2 \end{bmatrix}$ and $h'=\begin{bmatrix} m_1 & X \\
0 &m_2' \end{bmatrix}$ for $X \in \M(3,F)$. Let \[M_{m_{1},m_{2}}^{ d,\beta}=\{X \in S(\beta) \mid \dim(\Ker(h-1))=d\}\] for $\beta \in F$.
Clearly,
\[\Ker(m_2-1)=\Ker(m_2'-1).\]
Hence for any $X \in \M(3,F)$, 
\[X\Ker(m_2-1)=X\Ker(m_2'-1)\]
and \[X\Ker(m_2-1) \cap \Image(m_1-1)=X\Ker(m_2'-1) \cap \Image(m_1-1).\]
In particular, for any $\beta \in F$, we have that 
\begin{equation}\label{M_d^h}M_{m_{1},m_{2}}^{ d,\beta}=M_{m_{1},m_{2}'}^{ d,\beta}.
\end{equation}
Therefore, 
\[\Theta_{N,\psi_A}(m)=\frac{1}{q^9}(R_1+R_2)\]
where
\begin{align*} R_1&=\sum_{d=1}^{6}\sum_{ X \in M_{m_{1},m_{2}}^{ d,0}}\Theta_{\theta}(h)\overline{\psi_A({m_1}^{-1}X)}\\
&= \sum_{d=1}^{6}(-1)^{6-1}(1-q) \cdots (1-q^{d-1})(\# M_{m_{1},m_{2}}^{ d,0})\overline{\psi_0(0)} \end{align*}
and \begin{align*} R_2 &=\sum_{\beta \in F^{\times}}\sum_{d=1}^{6}\sum_{X \in M_{m_{1},m_{2}}^{ d,\beta}}\Theta_{\theta}(h)\overline{\psi_A({m_1}^{-1}X)}\\
 &=\sum_{\beta \in F^{\times}} \sum_{d=1}^{6}(-1)^{6-1}(1-q) \cdots (1-q^{d-1})(\# M_{m_{1},m_{2}}^{ d,\beta}\overline{\psi_0(\beta)}) \end{align*}
Thus,  \[\Theta_{N,\psi_A}(m)=\frac{1}{q^9}\sum_{d=1}^{6}(-1)^{6-1}(1-q) \cdots (1-q^{d-1})(\# M_{m_{1},m_{2}}^{ d,0}+ \sum_{\beta \in F^{\times}}(\# M_{m_{1},m_{2}}^{ d,\beta})\overline{\psi_0(\beta})).\]
Similarly, \[\Theta_{N,\psi_A}(m')=\frac{1}{q^9}\sum_{d=1}^{6}(-1)^{6-1}(1-q) \cdots (1-q^{d-1})(\# M_{m_{1},m_{2}'}^{ d,0}+ \sum_{\beta \in F^{\times}}\# (M_{m_{1},m_{2}'}^{ d,\beta})\overline{\psi_0(\beta})).\]
Hence, it follows from equation \ref{M_d^h} that
\[\Theta_{N,\psi_A}(m)=\Theta_{N,\psi_A}(m').\]
\end{proof}

\begin{remark} Let $1 \leq i \leq j \leq 10$. To determine $\Theta_{N,\psi_A}(m)$ for $m \in T(i,j)$, it is enough to compute $\Theta_{N,\psi_A}(m)$ for $m \in E$. We illustrate this by an example.\\

Suppose that we want to compute the character value $\Theta_{N,\psi_A}(m)$ for $m \in T(3,7)$. From Proposition \ref{(x,7)}, it follows that $\Theta_{N,\psi_A}(m)=\Theta_{N,\psi_A}(k)$ for some $k \in T(3,1)$. By Theorem \ref{T(i,j)=T(j,i)}, we have $\Theta_{N,\psi_A}(k)=\Theta_{N,\psi_A}(x)$ for some 
$x\in T(1,3)$. Using Theorem \ref{Type(x,3)}, we have,
$\Theta_{N,\psi_A}(x)=\Theta_{N,\psi_A}(y)$
for some $y \in T(1,1)$. Thus, using Theorem \ref{Type(1,1)} we have 
\[\Theta_{N,\psi_A}(m)=(1-q)^2.\]
For clarity, we represent the chain of computations used to determine the character value of an element in $T(i,j)$ to the character value of an element in the fundamental set $E$ in the following way.
\[(3,7) \rightarrow (3,1) \rightarrow (1,3) \rightarrow (1,1).\]
\end{remark}

The table below summarizes the sequence of computations used to calculate $\Theta_{N,\psi_{A}}(m)$ for $m\in T(i,j)$, $i\leq j$.

\begin{landscape}
\captionof{table}{Sequence of computations for $\Theta_{N,\psi_A}(m)$}
\label{Character of pi(N,psiA)}
\scriptsize
\resizebox{2\textwidth}{!}{
    \begin{tabular}{| c | c | c | c | c | c | c | c | c | c | c | }
    \hline
    & &    & & & & & & & & \\
    & Type-1 & Type-2 & Type-3 & Type-4 & Type-5 & Type-6 & Type-7 & Type-8 & Type-9 & Type-10\\
   \hline 
   
   & & & & & & & & & &\\
   Type-1& $(1,1)$ & $(1,2)$ & $(1,3) \rightarrow(1,1)$ & $(1,4)$ & $(1,5) \to (1,1)$ & $(1,6) \rightarrow (1,1)$ & $(1,7) \rightarrow (1,1)$ & $(1,8) \rightarrow (1,2)$ & $(1,9) \rightarrow (1,1)$ & $(1,10) \rightarrow (1,2)$\\
      & & & & & & & & & &\\
    \hline
    
       & & & & & & & & & &\\
       Type-2 & - & $(2,2)$ & $(2,3) \rightarrow (2,1) \rightarrow (1,2)$ & $(2,4)$ & $(2,5) \rightarrow (2,1) \rightarrow (1,2)$ & $(2,6) \rightarrow (2,1) \rightarrow (1,2)$ & $(2,7) \rightarrow (2,1) \rightarrow (1,2)$ & $(2,8) \to (2,2)$ & $(2,9) \to (2,1) \to (1,2)$ &$(2,10) \to (2,2)$\\
          & & & & & & & & & &\\
\hline

   & & & & & & & & & &\\
    Type-3& - & - & $(3,3) \to (3,1) \to (1,3)\to (1,1)$ & $(4,3) \to (4,1) \to (1,4)$ & $(3,5) \to (3,1) \to (1,3) \to (1,1)$ & $(3,6) \to (6,3) \to (6,1) \to (1,6) \to (1,1)$ & $(3,7) \to (3,1) \to (1,3) \to (1,1)$ & $(3,8) \to (3,2) \to (2,3) \to (2,1) \to (1,2)$ & $(3,9) \to (9,3) \to (9,1) \to (1,9) \to (1,1)$ & $(3,10) \to (3,2) \to (2,3) \to (2,1) \to (1,2)$ \\
       & & & & & & & & & &\\
\hline

   & & & & & & & & & &\\
Type-4& - & - & - & $(4,4)$ & $(4,5) \to (4,1) \to (1,4)$ & $(4,6) \to (4,1) \to (1,4)$ & $(4,7) \to (4,1) \to (1,4)$ & $(4,8) \to (4,2) \to (2,4)$ & $(4,9) \to (4,1) \to (1,4)$ & $(4,10) \to (4,2) \to (2,4) $\\
   & & & & & & & & & &\\
\hline

& & & & & & & & & &\\
   Type-5& - & - & - & - & $(5,5) \to (5,1) \to (1,5) \to (1,1)$ & $(5,6) \to (6,5) \to (6,1) \to (1,6) \to (1,1)$ & $(5,7) \to (5,1) \to (1,5) \to (1,1)$ & $(5,8) \to (5,2) \to (2,5) \to (2,1) \to (1,2)$ & $(5,9) \to (9,5) \to (9,1) \to (1,9) \to (1,1)$ & $(5,10) \to (5,2) \to (2,5) \to (2,1) \to (1,2)$\\
      & & & & & & & & & &\\
    \hline

& & & & & & & & & &\\
   Type-6 & - & - & - & - &- & $(6,6) \to (6,1) \to (1,6) \to (1,1)$ & $(6,7) \to (6,1) \to (1,6) \to (1,1)$ & $(6,8) \to (6,2) \to (2,6) \to (2,1) \to (1,2)$ & $(6,9) \to (6,1) \to (1,6) \to (1,1)$ & $(6,10) \to (6,2) \to (2,6) \to (2,1) \to (1,2)$\\
      & & & & & & & & & &\\
    \hline
    
    & & & & & & & & & &\\
   Type-7& - & - & - & - & - & - & $(7,7) \to (7,1) \to (1,7) \to (1,1)$ & $(7,8) \to (7,2) \to (2,7) \to (2,1) \to (1,2)$ & $(7,9) \to (9,7) \to (9,1) \to (1,9) \to (1,1)$ & $(7,10) \to (7,2) \to (2,7) \to (2,1) \to (1,2)$\\
      & & & & & & & & & &\\
    \hline
       
       & & & & & & & & & &\\
       Type-8 & - & - & - & - &- &- & - & $(8,8) \to (8,2) \to (2,8) \to (2,2)$ & $(8,9) \to (9,8) \to (9,2) \to (2,9) \to (2,1) \to (1,2)$ &$(8,10) \to (8,2) \to (2,8) \to (2,2)$\\
          & & & & & & & & & &\\
\hline

 & & & & & & & & & &\\
   Type-9& - & - & - & - & - & - & - & - & $(9,9) \to (9,1) \to (1,9) \to (1,1)$ & $(9,10) \to (9,2) \to (2,9) \to (2,1) \to (1,2)$\\
      & & & & & & & & & &\\
    \hline
    
 & & & & & & & & & &\\
       Type-10 & - & - & - & - & - & - & - & - & - & $(10,10) \to (10,2) \to (2,10) \to (2,2)$\\
          & & & & & & & & & &\\
\hline
    \end{tabular} 
    }
    \end{landscape}




\section*{Acknowledgements}

We thank Professor Dipendra Prasad for suggesting this problem and for some helpful discussions. Research of Kumar Balasubramanian is supported by the SERB grant: MTR/2019/000358.

\bibliographystyle{amsplain}
\bibliography{Twisted-Jacquet-Module}
\end{document}